\documentclass[microtype]{gtpart}
\usepackage{tocloft}
\usepackage{mathtools}
\usepackage[all,cmtip]{xy}

\usepackage{hyperref}
\hypersetup{
colorlinks = true,
linkcolor = {blue},
urlcolor = {red},
citecolor = {blue}
}
\usepackage[scr=rsfs]{mathalpha}
\usepackage[utf8]{inputenc}   
\usepackage[T1]{fontenc}      
\usepackage[english]{babel} 

\newcommand{\Rect}{\mathbf{R}}

\newcommand{\Mass}{\mathbb{M}}
\newcommand{\N}{\mathbb{N}}

\newcommand{\Lip}{\mathrm{Lip}}

\newcommand{\dist}{\mathrm{dist}}

\newcommand{\spt}{\mathrm{spt}\hspace{0.01cm}}

\newcommand{\dV}{d_V\kern-1pt}
\newcommand{\dW}{d_W\kern-1pt}

\newcommand\res{\mathop{\hbox{\vrule height 7pt width .3pt depth 0pt
\vrule height .3pt width 5pt depth 0pt}}\nolimits}

\newcommand{\modtwo}{_{_{\mathbb{Z}_2}}\hspace{-0.1cm}}

\title{On smooth approximation of integral cycles mod 2}

\author{Gianmarco Caldini}
\givenname{}
\surname{}
\address{}
\email{}
\urladdr{}

\newtheorem{thm}{Theorem}[section]    
\newtheorem{corollary}[thm]{Corollary}
\newtheorem{pro}[thm]{Proposition}    
\newtheorem{lem}[thm]{Lemma}          
\theoremstyle{definition}
\newtheorem{definition}[thm]{Definition}    

\theoremstyle{definition}

\theoremstyle{remark}
\newtheorem{remark}[thm]{Remark}             
 
\newtheorem{lemma}[thm]{Lemma} 
\theoremstyle{definition}
\newtheorem{assumptions}[thm]{Assumptions}    

\usepackage{stmaryrd}
\usepackage{titlesec}
\titlespacing*{\subsubsection}{0pt}{0.5cm}{0.15cm}

\begin{document}

\begin{abstract}
We prove that every mod 2 integral cycle $T$ in a Riemannian manifold $\mathcal{M}$ can be approximated in flat norm by a cycle which is a smooth submanifold $\Sigma$ of nearly the same area, up to a singular set of codimension 3; in addition, this estimate on the singular set can be refined depending on the codimension of the cycle. Moreover, if the mod 2 homology class $\tau$ admits a smooth embedded representative, then $\Sigma$ can be chosen free of singularities. This article provides the unoriented version of the smooth approximation theorem for integral cycles in \cite{ABCD}.
\end{abstract}

\maketitle
\vspace{-1cm}
\setlength{\cftbeforesecskip}{9pt}
\tableofcontents

\setlength{\parskip}{0cm}
\setlength{\parindent}{0.5cm}

\section{Introduction}
\subsection{Main theorem and corollaries}

In this article we are interested in the question of whether it is possible to approximate any integral current mod 2 belonging to a $\mathbb{Z}_2$ homology class of a smooth Riemannian manifold with currents induced by smooth submanifolds; in particular, we continue the study started in \cite{ABCD}, providing the analog $-$ without imposing any orientability assumption and for integral cycles mod 2 $-$ of the smooth approximation theorems for integral cycles in oriented domains developed therein. More formally, our main result is the following.

\begin{thm}\label{t:1}
Let $\mathcal{M}$ be a connected smooth closed (not necessarily orientable) Riemannian manifold of dimension $m+n$. Let $\varepsilon > 0$, $\tau$ an $m$-dimensional homology class in $H_{m}(\mathcal{M}, \mathbb{Z}_2)$, and $T$ an integral cycle  mod 2 representing $\tau$. Then there is a smooth triangulation $\mathcal{K}$ of $\mathcal{M}$ and an $m$-dimensional smooth submanifold $\Sigma$ of $\mathcal{M}\setminus \mathcal{K}^{m-n-1}$ (where $\mathcal{K}^{m-n-1}$ denotes $(m-n-1)$-skeleton of $\mathcal{K}$) with the following properties.
\begin{enumerate}
\item The $m$-dimensional volume of $\Sigma$ does not exceed the mass of $T$ by more than $\varepsilon$, that is $\mathcal{H}^{m}(\Sigma) \leq \mathbb{M}(T) + \varepsilon$.
\item The integral mod 2 cycle $\llbracket \Sigma \rrbracket\modtwo$ is homologous to $T$ and there is an $m+1$-dimensional integral mod 2 current $S$ in $\mathcal{M}$ such that $\partial S= \llbracket \Sigma \rrbracket\modtwo  - T$ and $\mathbb{M}(S)< \varepsilon$.
\item If $\tau$ admits a smooth embedded representative, then $\Sigma$ can be chosen to be a smooth submanifold of $\mathcal{M}$.
\end{enumerate}
\end{thm}

\begin{remark}
The codimension $n+1$ estimate on the singular set in Theorem \ref{t:1} is sharp in full generality, see Theorem \ref{t:sharp}; in fact, proving that the construction in Theorem \ref{t:1} for mod 2 homology is optimal is subtler than in the case of integral homology, \emph{cfr.} Section \ref{s:optimality}.
\end{remark}

\begin{remark}\label{r:St_pi} By the foundational work of Thom \cite{Thom54}, any mod 2 homology class $\tau \in H_m(\mathcal{M},\mathbb{Z}_2)$ is representable by a smooth embedded submanifold when $m=1,2,3$ or when $n=1$; likewise, this is true every time $m\leq n$, \textit{cfr.} \cite[Théorème II.26]{Thom54}. The lowest dimensional example not covered by Thom's algebraic computations is the 4-dimensional $\mathbb{Z}_2$ homology group in a 6-dimensional closed smooth manifold $\mathcal{M}$; in fact, it is a corollary of a construction due to Teichner, see \cite{Teichner95}, the existence of a 6-dimensional manifold $\mathcal{M}$ with a 4-dimensional homology class $\tau \in H_4(\mathcal{M},\mathbb{Z}_2)$ which cannot be represented by an embedded $-$ in fact, not even immersed, \emph{cfr.} \cite{Grantszucs} $-$ smooth submanifold, see Section \ref{s:optimality}.
\end{remark}

Since every mod 2 homology class of codimension $n=1$ is representable by a smooth embedded submanifold, a straightforward application of the same techniques developed in \cite{ABCD} in the integral setting would provide the following result: any integral cycle mod 2 can be approximated in the sense of Theorem \ref{t:1} with a smooth submanifold up to a singular set of codimension 3, regardless of the codimension $n\geq 2$. The key difference in the mod 2 setting is based on the better-behaved topological structure of the Thom space of the universal $n$-plane bundle over the infinite Grassmannian $BO(n)$, allowing for a refined result and improving the estimate on the singular set based on the codimension $n$ of the cycle, see Theorem \ref{t:productEM} and Remark \ref{r:productEM}.

\begin{remark}
We recall that every mod 2 homology class $\sigma \in H_m(\mathcal{M},\mathbb{Z}_2)$ admits a \emph{Steenrod representation}, \textit{i.e.} there exists a smooth manifold $\Sigma$ and a continuous map $f:\Sigma \rightarrow \mathcal{M}$ such that the fundamental class of $\Sigma$ equals $\sigma$, \emph{cfr.} \cite[Théorème III.2]{Thom54}. This differs substantially from the integral setting, where there exist integral homology classes that cannot be represented in the sense of Steenrod, \textit{cfr.} \cite[Théorème III.9]{Thom54}. For this reason, the problem of finding smooth embedded representatives in mod 2 homology classes can be understood as a problem of whether a continuous map is homotopic to an embedding.
\end{remark}

We will follow definitions and terminology of \cite{ABCD} and \cite[4.2.26]{Federerbook}, which will be briefly recalled in Section \ref{s:notationandpreliminaries}. Moreover, we will always rely on the following assumptions (unless otherwise stated).

\begin{assumptions}\label{a:1}
 $\mathcal{M}$ is a connected smooth closed (not necessarily orientable) Riemannian manifold of dimension $m+n$, where $n,m \in \N\setminus \{0\}$ are arbitrary positive integers, $\tau$ is an element of the $m$-dimensional homology group $H_{m}(\mathcal{M}, \mathbb{Z}_2)$ and $T$ is an integral mod 2 current (hence a cycle) representing $\tau$.
\end{assumptions}

In dealing with smooth triangulations of $\mathcal{M}$ we will avoid referring to the map $t:|\mathcal{K}| \rightarrow \mathcal{M}$ from the geometric realization of some fixed simplicial complex $\mathcal{K}$ and we will use directly $\mathcal{K}$ also for the smooth triangulation of $\mathcal{M}$. As usual, $\mathcal{K}^j$ will denote the $j$-dimensional skeleton of $\mathcal{K}$ and $\Sigma$ will denote either smooth $m$-dimensional closed embedded (not necessarily orientable) submanifolds of $\mathcal{M}$ or smooth $m$-dimensional embedded submanifolds of $\mathcal{M}\setminus \mathcal{K}^j$ (for some integer $j$) whose topological closure is contained in $\mathcal{K}^j$; We will denote by $\llbracket \Sigma \rrbracket\modtwo$ the integral cycle mod 2 induced by $\Sigma$. 

By Nash’s isometric embedding theorem we consider $\mathcal{M}$ as a submanifold of some Euclidean space $\mathbb{R}^N$. For every $k=0,\dots,m+n$, we denote by $\mathcal{Z}_k(\mathcal{M}, \mathbb{Z}_2)$ the set of $k$-dimensional integral mod 2 cycles with mod 2 support in $\mathcal{M}$ and by $\mathcal{Z}_{k, Lip}(\mathcal{M},\mathbb{Z}_2)$ the set of $k$-dimensional mod 2 Lipschitz cycles with mod 2 support in $\mathcal{M}$, \emph{i.e.} the set of mod 2 integral $k$-cycles of the form $f_{\#}(P)$ where $f: \mathbb{R}^N \rightarrow \mathcal{M}$ is a Lipschitz map and $P$ is a polyhedral cycle in $\mathbb{R}^N$.

\begin{definition}[\emph{Smooth representability}]\label{d:smoothrepresentability}
Let $\mathcal{M}, \tau$ and $\Sigma$ be as in Assumption \ref{a:1}. We say that $\tau$ is \emph{representable by a smooth submanifold} (or that $\tau$ \emph{admits a smooth representative}) if there exists a smooth embedding $f:\Sigma \rightarrow \mathcal{M}$ such that $f_*[\Sigma]=\tau$, where $[\Sigma]\in H_m(\Sigma, \mathbb{Z}_2)$ is the fundamental class of $\Sigma$. Analogously, denoting by $x \in H^n(\mathcal{M},\mathbb{Z}_2)$ the Poincaré dual of $\tau$, we say that $x$ is realized by a smooth submanifold whenever $\tau$ is.
\end{definition}

In analogy with \cite{ABCD}, an immediate corollary of Theorem \ref{t:1} is the absence of the \emph{Lavrentiev gap phenomenon} for the homological unoriented Plateau problem in absence of topological ostructions to realizability of mod 2 homology classes.

\begin{thm}[Absence of Lavrentiev gaps]\label{t:Lavrentiev}
Let $\mathcal{M}, \tau, T$ and $\Sigma$ as in Assumption \ref{a:1}, and define the following quantities:
\begin{align*}
&\mathbb{M}_{T}:=\min\{\mathbb{M}(T) : T \in \mathcal{Z}_m(\mathcal{M}, \mathbb{Z}_2) \cap \tau\},\\
&\mathbb{M}_{P}:=\inf\{\mathbb{M}(P) : P \in \mathcal{Z}_{m, Lip}(\mathcal{M}, \mathbb{Z}_2) \cap \tau\},\\
&\mathbb{M}_{\Sigma}:=\inf\{\text{Vol}\,^m(\Sigma) : \llbracket\Sigma\rrbracket\modtwo \in \tau \textrm{ and $\Sigma$ is smooth in $\mathcal{M} \setminus \mathcal{K}^{m-n-1}$ for some triangulation $\mathcal{K}$}\},\\
&\mathbb{M}_{{\rm Reg}} := \inf\{\text{Vol}\,^m(\Sigma) : \llbracket\Sigma\rrbracket\modtwo \in \tau \textrm{ and $\Sigma$ is smooth in $\mathcal{M}$}\}\, .
\end{align*}
Then, $\mathbb{M}_{T}= \mathbb{M}_{P}= \mathbb{M}_{\Sigma}$ and, moreover, $\mathbb{M}_T = \mathbb{M}_{{\rm Reg}}$ when $\tau$ is representable by a smooth submanifold.
\end{thm}

\begin{remark}
We remark that for an integral mod 2 cycle $T \in \mathcal{Z}_m(\mathcal{M}, \mathbb{Z}_2)$ its mod 2 mass $\mathbb{M}(T)$ coincides with its \emph{size}, that is the $m$-dimensional Hausdorff measure $\mathcal{H}^m(R)$ of the corresponding rectifiable set $R$.
\end{remark}

Another simple corollary of Theorem \ref{t:1} is the following approximation theorem with integral mod 2 cycles of prescribed singularities.

\begin{thm}[Approximation by cycles with prescribed singular sets]\label{t:prescribed}
Let $\mathcal{M}, \tau$ and $T$ be as in Assumption \ref{a:1}. Then, there is a sequence of smooth triangulations $\mathcal{K}_j$ of $\mathcal{M}$ and a sequence of smooth embedded  $m$-dimensional submanifolds $(\Sigma_j)_j$ in $\mathcal{M}\setminus \mathcal{K}_j^{m-n-1}$ such that
\begin{itemize}
\item[(a)] $\llbracket \Sigma_j\rrbracket\modtwo \rightarrow T$ in the mod 2 flat topology,
\item[(b)] $\lim_{j\rightarrow \infty}\mathcal{H}^{m}(\Sigma_j) = \mathbb{M}(T)$,
\item[(c)]$\partial \llbracket \Sigma_j\rrbracket\modtwo =0$ and $\llbracket \Sigma_j\rrbracket\modtwo$ is in the same homology class as $T$.
\end{itemize}
\end{thm}

\begin{remark}
If every mod 2 homology class of $\mathcal{M}$ admits a smooth embedded representative, then it is customary to call such manifolds $\mathbb{Z}_2$ \emph{totally realizable}. Examples of $\mathbb{Z}_2$ totally realizable manifolds are spheres $S^{k}$ and products of spheres $S^{k_1} \times S^{k_2} \times \ldots \times S^{k_i}$, or projective spaces $\mathbb{R}\mathbb{P}(n)$ and $\mathbb{C}\mathbb{P}(n)$ of real and complex dimension $n$ respectively, \emph{cfr.} \cite[Section 8]{Suzuki}. Clearly, in all such manifolds every mod 2 integral cycles can be approximated by smooth submanifolds by Theorem \ref{t:1}.
\end{remark}

\subsection{Motivation and overview of the proof}
\medskip
\noindent\textbf{Previous literature}
\vspace{0.1cm}

\noindent
Introduced in \cite{FedererFleming60} to provide a suitable framework to solve the oriented Plateau problem without topological or dimensional restrictions, Federer and Fleming's theory of integral currents necessarily requires the domain of integration to be oriented and homology classes to have integer coefficients; in order to generalize Federer and Fleming's theory to unoriented domains, it is possible to replace $\mathbb{Z}$ as coefficient group by the cyclic group $\mathbb{Z}_2$ of order 2: this has been done at the beginning of the sixties by Ziemer, later improved and generalized to any finite coefficient group by Fleming, see \cite{Ziemer, Fleming66}. Many regularity properties have been later derived for integral currents mod 2, both in terms of approximation theorems and in terms of \textit{a posteriori} regularity for mass-minimizers.

About the latter aspect of regularity theory, it is a corollary of Federer's dimension reduction argument, \textit{cfr.} \cite{Federer70}, that mod 2 mass-minimizing integral currents are induced by smooth submanifolds, up to an interior singular set of Hausdorff codimension at least 2, which turns out to be discrete for two-dimensional mod 2 currents \textit{cfr.} \cite{Almgren66}. In the particular case of codimension $n=1$, it is possible to derive finer properties of the singular set, showing that it is $m-7$-rectifiable and with locally finite $m-7$-Hausdorff measure as a consequence of the quantitative stratification theory by Naber and Valtorta, see \cite{NaberValtorta}. In codimension $n\geq 2$, it is a theorem of Simon the $m-2$-rectifiability of the singular set, \textit{cfr.} \cite[Corollary 1]{Simon93}, which is also now known to be of locally finite $m-2$-Hausdorff measure by again the methods of Naber and Valtorta in \cite{NaberValtorta}. 

About the former aspect of regularity, instead, the only \emph{a priori} approximation-type result for mod 2 integral currents was the classical mod 2 deformation theorem, \textit{cfr.} \cite[Theorem 4.2]{Ziemer} or \cite[(4.2.9)$^{\nu}$]{Federerbook} and, as a corollary, the mod 2 strong polyhedral approximation theorem, \textit{cfr.} \cite[(4.2.20)$^{\nu}$, (4.2.21)$^{\nu}$]{Federerbook}, both developed as simple adaptations of the ones for integral currents. A natural question is then to ask whether it is possible to approximate any intergral mod 2 current representing a $\mathbb{Z}_2$ homology class of a $-$ possibily non-orientable $-$ smooth closed Riemannian manifold by a smooth embedded submanifold; as already mentioned, there exist mod 2 homology classes which do not admit any smooth embedded representative: hence, in general the answer is negative. In analogy with \cite{ABCD}, this article provides an affirmative answer to the previous question every time the mod 2 homology class admits a smooth embedded representative and, in particular, it provides sharp estimates in full generality on the singular sets of the approximating sequence.

\vspace{0.25cm}
\noindent\textbf{Differences with integral homology and new constructions}
\vspace{0.1cm}

\noindent
We describe here the main ideas proof, whose construction closely follows the approach introduced in \cite{ABCD}. In particular, we emphasize its main differences from integral homology, which are mostly of topological nature and allow for the development of a finer argument than the one used in \cite{ABCD}. This yields a significantly better estimate for the singular set, which further improves with higher codimensions and which is due to the particularly well-behaved homotopy type of the Thom space of the (unoriented) $n$-plane bundle, see Theorem \ref{t:productEM} and Corollary \ref{c:restriction}. The key geometric measure theory ingredients $-$ like the \emph{codimension 2} smooth approximation theorem \cite[Proposition 4.1]{ABCD} and the \emph{squeezing} lemma \cite[Proposition 4.3]{ABCD} $-$ are, instead, easily adapted from the integral setting and hold for integral mod 2 cycles with minor modifications, see Section \ref{s:GMT}. Moreover, it is worth mentioning that proving sharpness of the codimension $n+1$ singular set of Theorem \ref{t:1} is even subtler in mod 2 homology than in integral homology, since singularities that appear in this context all arise from the impossibility of finding embeddings in low codimensions, and not $-$ as for integral classes $-$ by innate singularities obstructing Steenrod representability also; in particular, we cannot exploit Sullivan's geometric theory of resolution of singularities by blow-up in \cite{SullivanLiverpool}, as done in \cite[Theorem 6.3]{ABCD}. Instead, we need to rely on the singularity theory of stable mappings, coupled with an elegant result due to Grant and Sz\H{u}cs \cite{Grantszucs} about obstructions to realizability of mod 2 homology classes by immersions, see Section \ref{s:optimality}; we refer to \cite{GGbook} for an accessible overview on the theory of stable mappings.\vspace{0.25cm}

\vspace{0.1cm}
\noindent\textbf{Outline of the proof}
\vspace{0.1cm}

Starting from an integral mod 2 cycle $T$ representing a $\mathbb{Z}_2$ homology class $\tau$ of $\mathcal{M}$, we first approximate it with an integral mod 2 cycle $P'$ which is a smooth submanifold out of a small $\delta$-neighborhood $B_\delta$ of the $m-2$-skeleton of some smooth triangulation of $\mathcal{M}$, \emph{cfr.} Proposition \ref{p:poly_approx_prescribedsing}; in particular, a subset of the smooth part of $P'$ is a compact smooth submanifold with boundary embedded in a compact manifold with boundary denoted $\Omega$ (as in \cite{ABCD}, $\Omega$ can be thought as $\mathcal{M} \setminus B_\delta$, up to taking smooth neighborhoods of skeleta described in Section \ref{s:smoothingneighborhood}). This smooth part represents a relative homology class in $H_m(\Omega, \partial \Omega, \mathbb{Z}_2)$ and by the relative Pontryagin-Thom construction, \emph{cfr.} Theorem \ref{t:Thomboundary}, it induces a map $F: \Omega \rightarrow T(\gamma^n)$ with values in the Thom space of the universal $n$-plane bundle such that the pull-back of the Thom class is the Poincaré dual of $\tau$, once restricted to $\Omega$.

At this point, had we used the strategy adopted in the integral setting, we would now have simply been able to exploit the $n+2$-equivalence between $T(\gamma^n)$ and the Eilenberg-MacLane space $K(\mathbb{Z}_2,n)$ to prove that the restriction of the Poincaré dual of $\tau$ to the complement of a small neighborhood $U_\delta$ of the $(m-3)$-skeleton of the smooth triangulation of $\mathcal{M}$ admits a lift to $T(\gamma^n)$ pulling back the Thom class to itself; this would have provided an integral mod 2 cycle in $\tau$ given by a closed smooth embedded submanifold with singularities all contained in the $(m-3)$-dimensional skeleton $\mathcal{K}^{m-3}$ of $\mathcal{M}$. 

Instead, the homotopy type of $T(\gamma^n)$ behaves much better in the mod 2 case since, for every $n\geq 1$, $T(\gamma^n)$ is $2n$-equivalent to a complex $Y$ which is a \emph{product} of Eilenberg-Mac Lane spaces $K(\mathbb{Z}_2,n)$. This nicer structure allows us to exploit the $2n$-equivalence $H:T(\gamma^n) \rightarrow Y$ to obtain a map $g$ from the $2n$-skeleton of $K(\mathbb{Z}_2,n)$ to $T(\gamma^n)$ pulling-back the Thom class to the fundamental class of $K(\mathbb{Z}_2,n)$, simply by taking the restriction to the first factor of the inverse map of $H$, \emph{cfr.} Corollary \ref{c:restriction}. Nevertheless, we remark that it is still crucial for us to rely on the fact that the map $h: T(\gamma^n) \rightarrow K(\mathbb{Z}_2,n)$ itself is an $n+2$-equivalence for any $n\geq1$, \emph{cfr.} Lemma \ref{l:n+2_equivalenza}, since this allows us to perform the last part of the argument of the whole proof.

Hence, denoting by $Q$ the complement of a small neighborhood $U_\delta$ of the $(m-n-1)$-skeleton of the smooth triangulation of $\mathcal{M}$, it is possible to note that $Q$ has the homotopy type of a $2n$-dimensional skeleton of $\mathcal{M}$, \textit{cfr.} Lemma \ref{l:co-skeleton}. Thus, we exploit the $2n$-equivalence between $T(\widetilde{\gamma}^n)$ and the product $Y$ of Eilenberg-MacLane space $K(\mathbb{Z}_2,n)$, \textit{cfr.} Theorem \ref{t:productEM}, to obtain a map $g: K(\mathbb{Z}_2,n)^{2n} \rightarrow T(\gamma^n)$ pulling-back the Thom class to the fundamental class of $K(\mathbb{Z}_2,n)$ as described above; this allows us to prove that the restriction of the Poincaré dual of $\tau$ to $Q$ admits a lift \begin{equation}\label{e:riassuntof}f:Q \rightarrow T(\gamma^n)\end{equation} pulling back the Thom class to itself. Applying Theorem \ref{t:Thomboundary} and after some technicalities, this provides an integral mod 2 cycle $R$ homologous to $\tau$ which is a closed smooth embedded submanifold with singularities all contained in the $(m-n-1)$-dimensional skeleton $\mathcal{K}^{m-n-1}$ of $\mathcal{M}$.

Since, by Lemma \ref{l:co-skeleton}, $\Omega$ has the homotopy type of an $n+1$-dimensional complex, the $n+2$-equivalence between $T(\gamma^n)$ and $K(\mathbb{Z}_2,n)$ of Lemma \ref{l:n+2_equivalenza} allows us to conclude that homotopy classes of maps defined on $\Omega$ and with values in $T(\gamma^n)$ are in one-to-one correspondence with those with values in $K(\mathbb{Z}_2,n)$, \textit{cfr.} Corollary \ref{c:Whitehead_Homotopyclasses}. Hence, we can conclude that the smooth part of $P'$ coincides, up to a homotopy, with the smooth part of $R$ restricted to $\Omega$. 

The conclusion then follows in analogy with \cite{ABCD} from a technical geometric measure theory \emph{squeezing} lemma, \textit{cfr.} Proposition \ref{p:squash}, saying that if two $m$-dimensional integral mod 2 cycles $P'$ and $R$ agree outside of a sufficiently small neighborhood of the $(m-2)$-skeleton of a smooth triangulation of $\mathcal{M}$, then we can find a smooth deformation $R'$ of $R$ which is almost coinciding with $R$ and with mass close to the mass of $P'$. This provides the desired approximation $R'$ of Theorem $\ref{t:1}$, satisfying $(1)$ and $(2)$.

Finally, part $(3)$ of Theorem $\ref{t:1}$ is proved following the same lines: under the additional assumption that $\tau$ is representable by a smooth embedded submanifold we immediately obtain a map $$\ell:\mathcal{M}\rightarrow T(\gamma^n)$$ which pulls-back the Thom class to the Poincaré dual of $\tau$; the analogous construction can thus be performed just by replacing the map $f$ in \eqref{e:riassuntof} with $\ell$.

The rest of the paper is organized as follows. In Section \ref{s:notationandpreliminaries} we briefly recall the main notation and results in the theory of integral currents mod 2, in homotopy theory and cobordism. In Section \ref{s:smoothingneighborhood} we recall from \cite{ABCD} some technical preliminary lemmas about neighborhoods of skeleta and maps associated to them. In Section \ref{s:GMT} we state the adapted mod 2 version of the two main technical propositions from geometric measure theory, \emph{i.e.} Proposition \ref{p:poly_approx_prescribedsing} and Proposition \ref{p:squash}, and we comment on their proofs. Finally, Section \ref{s:proofofmain} is dedicated to the proof of Theorem \ref{t:1} and Section \ref{s:optimality} shows the optimality of the construction.

\subsection*{Acknowledgements}
I am very grateful to Sylvain Cappell for many enjoyable conversations and the NYU Courant Institute of Mathematical Sciences for its kind hospitality; I also thank Camillo De Lellis for his support, and Riccardo Ghiloni for mentioning Teichner's example in \cite{Teichner95}. This research has been funded by the Foundation Blanceflor Boncompagni Ludovisi, née Bildt, whose sponsorship is particularly acknowledged.

\section{Notation and preliminary results}\label{s:notationandpreliminaries}
We recall here the main definitions and relevant notation.

\subsection{Theory of integral currents mod 2}
To set up the main terminology, we briefly recall the theory of integral mod 2 currents as developed by Ziemer and completed (and generalized) by Fleming in \cite{Ziemer, Fleming66}. As a standard reference we also refer to \cite[4.2.26]{Federerbook}. For an introduction to the theory of integral currents we refer to \cite{FedererFleming60, Federerbook, Simonbook} or \cite[Section 2.1]{ABCD}.

We denote by $\mathbf{R}_k(\mathbb{R}^N)$ the space of $k$-dimensional integer rectifiable currents, by $\mathbf{I}_k(\R^N)$ the subgroup of $k$-dimensional integral currents and by $\mathbf{F}_k(\mathbb{R}^N)$ the space of $k$-dimensional integral flat chains in $\mathbb{R}^N$, that is $$\mathbf{F}_k(\mathbb{R}^N):=\{T + \partial S : T \in \mathbf{R}_k(\mathbb{R}^N), \,S \in \mathbf{R}_{k+1}(\mathbb{R}^N)\}.$$ If $M \subset \mathbb{R}^N$ is a compact Lipschitz neighborhood retract, then $\mathbf{R}_k(M)$ will denote the space of $T \in \mathbf{R}_k(\mathbb{R}^N)$ with support $\spt(T)\subset M$; analogously for $\mathbf{F}_k(M)$ and $\mathbf{I}_k(M)$.

\noindent
For $T \in\mathbf{F}_m(M)$, the (integral) flat norm of $T$ is defined as $$\mathbb{F}(T):=\inf \{ \Mass(R) + \Mass(S) \mid T=R +\partial S, \, R \in \mathbf{I}_k(M),\, S \in \mathbf{I}_{k+1}(M)\}.$$

\noindent
For $T \in \mathbf{F}_k(M)$ we recall that its mod 2 flat norm can be defined as $$\mathbb{F}^2(T):= \inf\{\mathbb{F}(T+2Q) : Q \in \mathbf{F}_k(M)\},$$ its mod 2 mass as $$\mathbb{M}^2(T):= \inf\{\mathbb{M}(T + 2Q) : Q \in \mathbf{F}_k(M)\}$$ and its mod 2 support as $$\spt^2(T):= \bigcap \{\spt(R) : R \in \mathbf{F}_k(M), \mathbb{F}^2(T-R)=0\}.$$ 

\noindent
We will usually drop the index 2 when this is clear from the context. The space of flat chains mod 2 in $M$ is defined as the quotient group $\mathbf{F}_k(M,\mathbb{Z}_2):=\mathbf{F}_k(M)/2\mathbf{F}_k(M)$ and if $T \in \mathbf{F}_k(M)$, we will denote with $[T]$ or $T$ mod 2 the coset in $\mathbf{F}_k(M,\mathbb{Z}_2)$ containing $T$. Analogously, we can define the space of integer rectifiable currents mod 2 and of integral currents mod 2, respectively denoted by $$\mathbf{R}_k(M,\mathbb{Z}_2):=\mathbf{R}_k(M)/2\mathbf{R}_k(M), \quad \mathbf{I}_k(M,\mathbb{Z}_2):=\mathbf{I}_k(M)/2\mathbf{I}_k(M).$$ 

\noindent
If $T$ is flat chain mod 2 with an integer coefficient polyhedral chain as a representative, then $T$ will be called a polyhedral chain mod 2; analogously for Lipschitz chains mod 2.

In particular, we remark that an integer rectifiable current $T =\llbracket E,\tau, \theta \rrbracket \in \Rect_k(M)$ is called \emph{representative mod 2} whenever $|\theta(x)|\leq 1$ for $\|T\|$-$a.e.$ point. Hence, every integer rectifiable current can be written as $T+2Q$, with $T, Q \in \Rect_k(M)$ and $T$ is a representative mod 2. In particular, $\Mass^2([T])= \Mass(T)$ and $\spt^2([T])=\spt(T)$, whenever $T \in \Rect_k(M)$ and $T$ is a representative mod 2. If $T \in \Rect_k^2(M)$ and $B \subset \mathbb{R}^N$ is a Borel subset, we define $T\res A:= [T\res A]$, where $T$ is a rectifiable representative of $T$ mod 2.

If $T \in \mathbf{F}_k(M)$ is a flat chain, then the boundary operator $\partial^2$ on $\mathbf{F}^2_k(M)$ is defined by $\partial^2 [T]=[\partial T]$, since if $T=S$ mod 2, then also $\partial T=\partial S$ mod 2. A flat chain $T \in \mathbf{F}_k(M)$ is a cycle mod 2 if $\partial T=0$ mod 2, \emph{i.e.} $\partial^2 [T]=0$ and it is a boundary mod 2 if there exists $S \in \mathbf{F}_{k+1}(M)$ such that $T=\partial S$ mod 2, \emph{i.e.} $[T]=\partial^2 [S]$; the spaces of $k$-dimensional integral mod 2 cycles and $k$-dimensional integral mod 2 boundaries with $\spt^2(\cdot) \subset M$ are denoted by $\mathcal{Z}_k(M,\mathbb{Z}_2)$ and $\mathcal{B}_k(M,\mathbb{Z}_2)$, respectively; note that $\mathcal{B}(M,\mathbb{Z}_2)$ is a normal subgroup of $\mathcal{Z}(\mathcal{M},\mathbb{Z}_2)$. If $\mathcal{M}$ is a closed smooth manifold, we define its mod 2 homology groups as $$H_*(\mathcal{M},\mathbb{Z}_2)= \mathcal{Z}_*(\mathcal{M},\mathbb{Z}_2)/\mathcal{B}_*(M,\mathbb{Z}_2),$$ and it is possible to prove, as a consequence of the deformation theorem, that the homology theory so defined satisfies all seven Eilenberg-Steenrod Axioms, and hence it is equivalent to the singular homology of $\mathcal{M}$ with coefficients in $\mathbb{Z}_2$. In the following sections, we will slightly abuse notation between currents mod 2 and their representatives. Moreover, when the coefficients are not specified, homology and cohomology have to be understood with coefficients in $\mathbb{Z}_2$.

\subsection{Homotopy theory and cobordism}
We briefly recall the main topological notions that will be used later; we also refer to \cite{Spanier, Switzer, Thom54, MilnorStasheff} and to \cite{ABCD}.

For $n \geq 1$ and an abelian group $\pi$, the \emph{Eilenberg-MacLane space} $K(\pi, n)$ is a space with the homotopy type of a $CW$-compex such that $\pi_i(K(\pi,n))$ vanishes for $i \neq n$ and $\pi_n(K(\pi,n))\simeq \pi$, where $\pi_i (X)$ denotes the $i$-th homotopy group of the topological space $X$. Recall that $K(\pi, n)$ is the classifying space of $n$-dimensional cohomology with coefficients in $\pi$, meaning that for a (connected) $CW$-complex $X$, an abelian group $\pi$ and for every $n \in \mathbb{N}\setminus \{0\}$ there is a natural isomorphism $$T: [X, K(\pi, n)] \rightarrow H^n(X,\pi),$$ where $[X, K(\pi, n)]$ represents the set of (unbasedpointed) homotopy classes of continuous maps from $X$ to $K(\pi, n)$ and $H^n (X, \pi)$ is the $n$-th cohomology group of $X$ with coefficients in $\pi$.

Given a continuous map $f: X \rightarrow Y$ between topological spaces we denote by $M_f$ its mapping cylinder, that is the quotient space formed from the disjoint union $(X \times [0,1]) \sqcup Y$ by identifying, for each $x \in X$, the point $(x, 1)$ with $f(x) \in Y$; it contains $X \times\{0\}$ as a subspace and has $Y$ as a deformation retract. We also recall that a continuous map $f:X \rightarrow Y$ between path-connected $CW$-complexes is an \emph{$n$-equivalence} for $n\geq1$ if the induced homomorphism $$f_* : \pi_i(X) \rightarrow \pi_i(Y)$$ is an isomorphism for $0<i<n$ and an epimorphism for $i=n$. In particular $f:X \rightarrow Y$ is an $n$-equivalence if and only if the inclusion $i:X \rightarrow M_f$ is an $n$-equivalence. From the long exact sequence of relative homotopy groups, it follows that $i$ is an $n$-equivalence if and only if the relative homotopy group $\pi_i(M_f,X)=0$ vanishes for all $i\leq n$. As usual, with a slight abuse we will sometimes write $\pi_i(Y,X)=0$, meaning $\pi_i (M_f,X)=0$ when the map $f$ is clear from the context. 

We state here a few facts that will be useful later on in the proof.


\begin{corollary}[\protect{\cite[Corollary 7.6.23]{Spanier}}]\label{c:Whitehead_Homotopyclasses}
If $K$ is a $CW$-complex and $f: X \rightarrow Y$ is an $n$-equivalence, then the induced homomorphism $f_*:[K, X] \rightarrow[K, Y]$ is a bijection if $\operatorname{dim} K<n$ and a surjection if $\operatorname{dim} K=n$. 
\end{corollary}

\begin{thm}[\protect{\cite[Theorem II.6]{Thom54}}]\label{t:II.6}
    Let $f: X \rightarrow Y$ a map between two simply connected $CW$-complexes $X,Y$. For $p$ prime, if for any group coefficient $\mathbb{Z}_p$ the induced homomorphism $$f^*: H^i(Y,\mathbb{Z}_p) \rightarrow H^i(X,\mathbb{Z}_p)$$ is an isomorphism when $i<k$ and a monomorphism when $i=k$, then the relative homotopy groups $\pi_i(Y,X)=0$, for $i\leq k$ and $f$ is a $k$-equivalence.
\end{thm}

 We denote by $G_n(\mathbb{R}^{n+k})$ the Grassmannian manifold, that is the space of $n$-dimensional subspaces in $\mathbb{R}^{n+k}$. The natural embedding $\mathbb{R}^n \hookrightarrow \mathbb{R}^{n+1}$, induces an embedding $G_n(\mathbb{R}^{n}) \hookrightarrow G_n(\mathbb{R}^{n+1})$ and thus, forming the union over increasing dimensions, we obtain an infinite $CW$-complex named the \emph{infinite Grassmannian} $$G_n:=G_n(\mathbb{R}^\infty)=\bigcup_k G_n(\mathbb{R}^{n+k}),$$ as the set of all $n$-dimensional linear subspaces of $\mathbb{R}^{\infty}$, endowed with the direct limit topology, \textit{i.e.} a subset of $G_n$ is open if and only if its intersection with $G_n(\mathbb{R}^{n+k})$ is open as a subset of $G_n(\mathbb{R}^{n+k})$ for each $k$. We denote $\gamma^n$ the \emph{universal n-plane bundle}, that is the canonical vector bundle over the base space $G_n$ $$E \xrightarrow{\pi} G_n$$ with total space $E$ consisting of pairs $(\ell, v) \in G_n(\mathbb{R}^{\infty}) \times \mathbb{R}^{\infty}$ such that $v \in \ell$, topologized as a subset of the cartesian product, and with projection $\pi:E \rightarrow G_n$ such that $\pi(\ell,v)=\ell$. Since the classifying space $G_n$ for $n$-plane vector bundles is the classifying space associated to the orthogonal group $O(n)$, we will denote it as usual by $BO(n)$. In fact, in almost all our considerations we just need to consider $G_n (\mathbb R^{n+k})$ for a sufficiently large $k$ and at all effects treat $BO(n)$ as some fixed compact manifold $G_n(\mathbb{R}^{n+k})$ for a suitably large $k$. 

Let $\xi$ be an $n$-plane bundle with a Euclidean metric and $A\subset E(\xi)$ the subset of the total space consisting of all vectors $v$ with $|v|\geq 1$. Recall that the identification space $E(\xi)/A$ is called the \emph{Thom space} $T(\xi)$ of $\xi$. Note that $T(\xi)$ has a preferred base point, denoted by $\infty$, and the complement $T(\xi) \setminus \{\infty\}$ consists of all vectors $v \in E(\xi)$ with $|v|<1.$ We note that if the base space $B$ of $\xi$ is a (finite) $CW$-complex, then the Thom space $T(\xi)$ is an $n-1$-connected (finite) $CW$-complex. If $\xi$ is a smooth $n$-plane bundle, then the base space $B$ of $\xi$ can be smoothly embedded as the zero-cross section in the total space $E(\xi)$, and hence in the Thom space $T(\xi)$; moreover we note that while $T(\xi)$ is not a manifold in general, the complement of the base point $T(\xi)\setminus \{\infty\}$ is a smooth manifold.

Let $R$ be a commutative ring with unity and $\xi$ an $n$-plane bundle $E \xrightarrow{\pi} B$. For a point $b \in B$, let $S_b^n$ be the one-point compactification of the fiber $\pi^{-1}(b)$; since $S_b^n$ is the Thom space of $\xi_{|b}$, we have the canonical inclusion map $i_b: S_b^n \rightarrow T (\xi)\,.$ An \emph{$R$-orientation}, or a \emph{Thom class}, of $\xi$ is defined to be an element of the reduced $n$-th cohomology group $u \in \widetilde{H}^n(T (\xi) , R)$ such that, for every point $b \in B$, $i_b^*(u)$ is a generator of the free $R$-module $\widetilde{H}^n(S_b^n)$. We recall now a fundamental theorem, \textit{cfr.} \cite[Theorem 15.51]{Switzer}.

\begin{thm}[Thom isomorphism theorem]
    Let $u \in \widetilde{H}^n(T (\xi), R)$ be a Thom class for an $n$-plane bundle $\xi$ of the form $E\xrightarrow{\pi} B$. Define $$ \Phi: H^i(B , R) \rightarrow \widetilde{H}^{n+i}(T (\xi), R)$$ by the cup product $\Phi(x)=\pi^*(x) \smile u$. Then $\Phi$ is an isomorphism for every integer $i$.
\end{thm}
\noindent
We remark that for every $n$-plane bundle there exists a unique Thom class with $\mathbb{Z}_2$ coefficients, \textit{cfr.} \cite[Theorems 8.1]{MilnorStasheff}. Thom's fundamental result about realizability of cycles by means of submanifolds can be stated as follows, \textit{cfr.} \cite[Théorème II.1]{Thom54}.

\begin{thm}\label{t:Thomclosed}
    Given $\mathcal{M}$ and $\tau$ as in Assumption \ref{a:1}, a homology class $\tau \in H_{m}(\mathcal{M}, \mathbb{Z}_2)$ is representable by a $m$-dimensional smooth submanifold $\Sigma \subset \mathcal{M}$ of codimension $n$ if and only if there exists a map $f: \mathcal{M} \rightarrow T(\gamma^n)$ which pulls back the Thom class $u \in H^n(T(\gamma^n), \mathbb{Z}_2)$ to the Poincaré dual of $\tau$. 
\end{thm}


By suitably modifying the proof of \cite[Théorème II.1]{Thom54}, it is possible to derive the analog of Theorem \ref{t:Thomclosed} for compact manifolds with boundary: for a proof we refer to \cite[Theorem 2.6]{ABCD}, which holds almost \emph{verbatim}, just replacing integral homology with homology mod 2, and $T(\widetilde{\gamma}^n), BSO(n)$ with $T(\gamma^n), BO(n)$ respectively.

\begin{thm}\label{t:Thomboundary}
Let $\mathcal{M}$ be a connected smooth compact $m+n$-dimensional Riemannian manifold with boundary $\partial \mathcal{M}$ and $\tau$ a relative homology class $\tau \in H_m(\mathcal{M},\partial \mathcal{M}, \mathbb{Z}_2)$. Then $\tau$ is representable\footnote{With the clear modifications in Definition \ref{d:smoothrepresentability} for $\mathcal{M},\Sigma$ with boundary and $\tau$ a relative mod 2 homology class.} by a smooth compact embedded submanifold manifold $\Sigma \subset \mathcal{M}$ with $\partial \Sigma = \Sigma\cap \partial \mathcal{M}$ if and only if there exists a map $f: \mathcal{M} \rightarrow T(\gamma^n)$ which pulls back the Thom class $u \in H^n(T(\gamma^n), \mathbb{Z}_2)$ to the relative Poincaré dual of $\tau$. 
\end{thm}

\section{Regular neighborhoods and maps}\label{s:smoothingneighborhood}

\subsection{Neighborhoods of skeleta}\label{s:subsmoothingneighborhood}

We list here a few technical tools introduced in \cite{ABCD} to deal with neighborhoods of skeleta of smooth triangulations and we recall the definition of some useful maps. We also refer to \cite{Hirschsmoothrn} for the theory of smooth regular neighborhoods.

For a fixed smooth triangulation $\mathcal{K}$ of $\mathcal{M}$ and a skeleton $\mathcal{K}^j$, we denote by $B_\delta (\mathcal{K}^j)$ the metric neighborhoods of the skeleton, \textit{i.e.} the sets of points $p$ with $\dist (p, \mathcal{K}^j) < \delta$. Fix a (sufficiently large) constant $C_0$ which will depend on the triangulation $\mathcal{K}$, subdivide the simplices forming $\mathcal{K}^j$ into $\mathcal{S}_0\cup \ldots \cup \mathcal{S}_j$ according to their dimension ($\mathcal{S}_i$ being the collection of simplices of dimension $i$) and set
\begin{equation}\label{e:intorni-spigolosi}
V_\delta (\mathcal{K}^j) := \bigcup_{i=0}^j \bigcup_{\sigma \in \mathcal{S}_i} B_{C_0^{-i} \delta} (\sigma)\, 
\end{equation}
where 
\[
B_{C_0^{-i} \delta} (\sigma) = \{p: \dist (p, \sigma) < C_0^{-i} \delta\}\, .
\]
As an elementary consequence of the definition, one obtains the following.

\begin{lemma}\label{l:spigoli}
For every smooth triangulation $\mathcal{K}$ of $\mathcal{M}$ and every $j\leq m+n-1$ there is a choice of $\bar \delta>0$ (sufficiently small) and of $C_0$ (sufficiently large) such that $\mathcal{M}\setminus V_{\bar\delta} (\mathcal{K}^j)$ is a deformation retract of $\mathcal{M}\setminus \mathcal{K}^j$ and such that for every point $p\in \partial V_{\bar\delta} (\mathcal{K}^j)$ there is at most one $\sigma$ in each $\mathcal{S}_i$ (with $0\leq i \leq j$) for which $p\in \partial B_{C_0^{-i} \bar\delta} (\sigma)$.

\noindent
In particular, there is a neighborhood of $U$ of $p$, an integer $\bar j\in \{1, \ldots, j\}$ and a diffeomorphism $\phi: U \to B_1 \subset \mathbb R^{m+n}$ (the unit ball in $\mathbb R^{m+n}$) such that 
\[
\phi (U\setminus V_{\bar\delta} (\mathcal{K}^j)) = \{(x_1, \ldots , x_{m+n}) : x_i>0 \mbox{ for $1\leq i \leq \bar{j}$}\}\, .
\]
\end{lemma}

\noindent
We also recall the appropriate regularization we need, and a key lemma about the homotopy type of complements of skeleta.

\begin{lemma}\label{l:spigoli-allisciati}
Let $\mathcal{K}$ be a smooth triangulation of $\mathcal{M}$, let $\bar \delta$ and $C_0$ be given by Lemma \ref{l:spigoli} and fix any pair of positive numbers $\delta' <\delta<\bar\delta$. Then there is a neighborhood $U_\delta (\mathcal{K}^j)$ of $\mathcal{K}^j$ with the following properties:
\begin{itemize}
\item $V_\delta (\mathcal{K}^j) \supset U_\delta (\mathcal{K}^j) \supset V_{\delta'} (\mathcal{K}^j)$;
\item The boundary of $U_\delta (\mathcal{K}^j)$ is smooth;
\item $\mathcal{M}\setminus U_\delta (\mathcal{K}^j)$ is a deformation retract of $\mathcal{M}\setminus V_{\delta'} (\mathcal{K}^j)$;
\item There is a smooth tubular neighborhood $\mathcal{C}$ of $\partial U_\delta (\mathcal{K}^j)$ in $\mathcal{M}$ containing $\partial V_{\delta'} (\mathcal{K}^j)$. 
\end{itemize}
\end{lemma}

\begin{lemma}\label{l:co-skeleton}
    Let $\mathcal{M}$ be as in Assumption \ref{a:1}, $\mathcal{K}$ a smooth triangulation of $\mathcal{M}$, $k\in \{0, \ldots , m+n-1\}$ and $U_\delta (\mathcal{K}^k)$ as in Lemma \ref{l:spigoli-allisciati}. Then the complement of $U_\delta(\mathcal{K}^k)$ is homotopy equivalent to a complex of dimension $m+n-k-1$.
\end{lemma}

\begin{proof}
See \cite[Lemma 5.1]{ABCD}.
\end{proof}

\subsection{Squeezing maps}

We associate to the neighborhoods $B_\delta , V_\delta$ some maps whose definition is recalled in the next lemma; we refer to \cite{ABCD} for a proof.

\begin{lem}\label{l:Phi}
Let $\mathcal{M}$ be as in Assumption \ref{a:1} and $\mathcal{K}$ a smooth triangulation of $\mathcal{M}$. For every $\varepsilon_a>0$ and $ \eta_a >0$ there is a positive number $\delta_a <  \eta_a$ with the following property. If $\gamma\in (0,1]$ is an arbitrary number, then there is a diffeomorphism $\Phi$ such that:
\begin{enumerate}\itemsep0.2em
\item $\Phi$ is isotopic to the identity and it coincides with the identity on $\mathcal{M}\setminus B_{\eta_a} (\mathcal{K}^k)$;
\item $\Lip (\Phi) \leq 1+ \varepsilon_a$;
\item For every point $p\in B_{\delta_a} (\mathcal{K}^k)$ there is an orthonormal frame $e_1, \ldots, e_{m+n}$ such that 
\begin{align}
|d\Phi_p (e_i)| &\leq 1+ \varepsilon_a \qquad &&\forall i\in \{1, \ldots, k\}, \tag{3.1}\\
|d\Phi_p (e_j)| &\leq \gamma &&\forall j\in \{k+1, \ldots, m+n\}\, . \tag{3.2}
\end{align}
\end{enumerate}
\end{lem}

\noindent
We also recall the following simpler lemma, obtained as a minor modification of the former.

\begin{lemma}\label{l:second-Phi}
Let $\mathcal{M}$ be as in Assumption \ref{a:1}, $\mathcal{K}$ a triangulation, $j\in \{0, \ldots, m+n-1\}$. If $C_0$ and $\bar\delta^{-1}$ in Lemma \ref{l:spigoli} are chosen sufficiently large, then for every $\delta_b> \delta'_b>0$ there is a Lipschitz map $\Phi: \mathcal{M}\to \mathcal{M}$ with the following properties:
\begin{itemize}
\item $\Phi$ maps $V_{\delta'_b} (\mathcal{K}^j)$ into $\mathcal{K}^j$;
\item $\Phi$ is a smooth diffeomorphism between $\mathcal{M}\setminus \Phi^{-1} (\mathcal{K}^j)$ and $\mathcal{M}\setminus \mathcal{K}^j$;
\item $\Phi (p)=p$ for every $p\not\in V_{\delta_b} (\mathcal{K}^j)$. 
\end{itemize}
\end{lemma}

\noindent
Finally we recall the following consequence of the area formula, that is easily adapted from the integral setting to integral currents mod 2, taking into account that if $[T]$ is an integral current mod 2, $\varepsilon >0$ and $U$ an open set containing $\spt [T]$, then there exists an integral current $S=T$ mod 2 such that $\spt(S) \subset U$ and $\Mass(S)< \Mass([T]) + \varepsilon$.

\begin{lem}\label{l:squash-2}
Assume $\Phi$, $\gamma$, and $\varepsilon_a$ are as in Lemma \ref{l:Phi}. If $Z$ is any integer rectifiable current mod 2 of dimension $m>k$, then
\begin{equation}\label{e:est-mass}
\mathbb{M} (\Phi_\sharp (Z\res B_{\delta_d} (\mathcal{K}^k)))\leq C (1+ \varepsilon_a)^k \gamma^{m-k} \|Z\| (B_{\delta_d} (\mathcal{K}^k))\, ,   
\end{equation}
where $C$ is a dimensional constant which depends only on $m$ and $n$. 
\end{lem}

\section{Geometric measure theory propositions}\label{s:GMT}

We state the mod 2 analogs of the main geometric measure theory tools needed in the proof of Theorem \ref{t:1}, which have independent interest. 

\subsection{Codimension 2 smooth approximation}

The first statement $-$ Proposition \ref{p:poly_approx_prescribedsing} $-$ is the mod 2 version of the \emph{codimension 2} smooth approximation theorem in \cite[Section 4]{ABCD}, that is an upgraded version of the Federer and Fleming's strong polyhedral approximation theorem in the spirit of \cite[$(4.2.21)^{\nu}$]{Federerbook}\footnote{See also \cite[Theorem 3.4]{Marchesestuvard}.}, where the approximands are smooth submanifolds out from the $m-2$-skeleton $\mathcal{K}^{m-2}$ of a smooth triangulation of $\mathcal{M}$. The proof is an easy adaptation of the one in \cite{ABCD} for integral cycles and we remark that since we are dealing with integral cycles mod 2, the integral current representative has multiplicity one (and possibly boundary of multiplicity 2) by definition; hence, the proof of Proposition \ref{p:poly_approx_prescribedsing} is even simpler than in the integral setting since there is no need to perform Step 1 in the \emph{second approximation} (\emph{i.e.} clearing out the multiplicity and regularizing $\mathcal{M}\setminus\mathcal{K}^{m-1}$).

\begin{pro} \label{p:poly_approx_prescribedsing}
Let $\mathcal{M}$ be as in Assumption \ref{a:1} and $T$ be an $m$-dimensional integral cycle mod 2 in $\mathcal{M}$. 
For every fixed $\varepsilon_c > 0$ there is a mod 2 integral cycle $P$ homologous to $T$ and a smooth triangulation $\mathcal{K}$ of $\mathcal{M}$ with the following properties:
\begin{itemize}
\item[($a_0$)] $\mathbb{M}(P) \leq (1+\varepsilon_c) \mathbb{M}(T)$;
\item[($b_0$)] $\mathbb{F} (T-P)\leq \varepsilon_c$;
\item[($c_0$)] $\spt(P) \subset \{ x : \dist(x,\spt(T))\leq \varepsilon_c\}$;
\item[($d_0$)] for every $p\in \spt (P) \setminus \mathcal{K}^{m-1}$ there is a neighborhood $U$ of $p$ and a smooth $m$-dimensional submanifold $\Lambda$ of $\mathcal{M} \cap U$ with boundary contained in $\mathcal{M} \cap \partial U$ and such that $P \res U = \llbracket \Lambda \rrbracket$. 
\end{itemize}
\vspace{0.25cm}
\noindent
Furthermore, for a sufficiently small $\delta_c'>0$ and any $\delta_c<\delta_c'$ we can find another mod 2 integral cycle $P'$ homologous to $P$ with the following properties:
\vspace{0.25cm}
\begin{itemize}
\item[(a)] $\mathbb{M} (P') \leq (1+2\varepsilon_c) \mathbb{M} (T)$ and $\mathbb{F} (T-P') \leq 2\varepsilon_c$;
\item[(b)] $\spt(P') \subset \{ x : \dist(x,\spt(T))\leq 2\varepsilon_c\}$;
\item[(c)] $\|P'\| (B_{\delta'_c} (\mathcal{K}^{m-2})) \leq 2\varepsilon_c$;
\item[(d)] $P'\res \mathcal{M}\setminus B_{\delta_c} (\mathcal{K}^{m-2})= \llbracket\Gamma \rrbracket\modtwo$ for some smooth submanifold $\Gamma$ of $\mathcal{M}\setminus B_{\delta_c} (\mathcal{K}^{m-2})$ without boundary in $\mathcal{M}\setminus B_{\delta_c} (\mathcal{K}^{m-2})$.
\end{itemize}
\end{pro}

\begin{proof}
Choose an integral mod 2 cycle $T$ in $\mathcal{M}$ and consider it as a mod 2 integral cycle in $\mathbb{R}^N$. Apply now the \emph{first approximation} of \cite[Proposition 4.1]{ABCD}, which can be readily adapted to mod 2 cycles, to obtain a mod 2 integral cycle $P$ homologous to $T$ with $\spt(P) \subset \mathcal{M}$ and such that 
\begin{itemize}
\item[($a_0$)] $\mathbb{M}(P) \leq (1+\varepsilon_c) \mathbb{M}(T)$;
\item[($b_0$)] $\mathbb{F} (T-P)\leq \varepsilon_c$;
\item[($c_0$)] $\spt(P) \subset \{ x : \dist(x,\spt(T))\leq \varepsilon_c\}$;
\item[($d_0$)] $P = \sum_{F\in \mathcal{F}^m} \llbracket F \rrbracket$ mod 2, where $\mathcal{F}^m$ is the collection of $m$-dimensional oriented cells of $\mathcal{K}$, for a suitable choice of $\mathcal{K}=\mathcal{K}(T, \varepsilon)$.
\end{itemize}

\noindent
Starting with $P$ we now need to perform the \emph{second approximation} of \cite[Proposition 4.1]{ABCD}, which was divided in Step 1 (regularization in the complement of $\mathcal{K}^{m-1}$) and Step 2 (removal of the $m-1$-dimensional singularities); note that since $P$ has multiplicity $1$ for $\|P\|$-a.e. point, $P$ is already regular on $\mathcal{M}\setminus \mathcal{K}^{m-1}$, with no need to perform Step 1. 

We now only need to perturb $P$ by removing its $m-1$-singularities away from a small neighbourhood of $\mathcal{K}^{m-1}$, which follows closely the procedure done in \cite{ABCD}; for the reader’s convenience, we recall it below with the needed minor adjustments.

\vspace{0.25cm}
\noindent
{\bf Removing the $m-1$-dimensional singularities.}

Consider an arbitrary face $F^k$ and let $\sigma_i$ be an arbitrary $(m-1)$-dimensional face of $F^k$. Fixing a $\delta_c>0$, we modify $P$ in a neighborhood of $\sigma\setminus B_{\delta_c} (\mathcal{K}^{m-2})$ to a new current $P'$ in the same homology class, close to it in terms of mass, support and flat norm, and with the property that $P'$ is smooth in that neighborhood. The neighborhood in which we perturb $P$ is of the form $B_\lambda (\sigma)\setminus B_{\delta_c} (\mathcal{K}^{m-2})$. 

Given the structure of $P$ just obtained, if $\lambda$ is sufficiently small, there is an open subset $U_i \subset \mathbb R^{m-1}$ and a smooth parametrization 
$$
\Phi : U_i \times B_{\lambda}^{n+1} \to \mathcal{M}
$$
of a normal neighborhood $\mathcal{N}_i$ of $\sigma \setminus B_{\delta_c} (\mathcal{K}^{m-2})$ with thickness $\lambda_i$ and with the property that $\spt (P\res \mathcal{N}_i)$ can be described in the following way. There are two distinct unit vectors $w_1, w_2 \in \partial B_1^{n+1} \subset \mathbb R^{n+1}$ such that, if we let 
$$
\Lambda_\ell = \{(\sigma, s w_\ell) : \sigma \in U_i, 0 < s < \lambda_i\}\, ,
$$
then, since $P$ has no mod 2 boundary in $\mathcal{N}_i$,

 $$P \res \mathcal{N}_i = \Phi_\sharp \llbracket \Lambda_1 \rrbracket + \Phi_\sharp \llbracket \Lambda_2 \rrbracket \,\, \text{mod 2}.$$

\noindent
Consider now the halflines $H_\ell = \{\lambda w_\ell : \lambda >0\}$ for $\ell\in \{1,2\}$ in $\mathbb R^{n+1}$. We can find a smooth curve $\gamma$ in $\mathbb R^{n+1}$ such that $\llbracket \gamma \rrbracket \res \mathbb R^{n+1} \setminus B_1= (\llbracket H_1 \rrbracket  - \llbracket H_{2}\rrbracket)\res \mathbb R^{n+1}\setminus B_1$ mod 2. Furthermore we let $\tau_t: \mathbb R^{n+1} \to \mathbb R^{n+1}$ be the homothety $y\mapsto t y$ and denote by $\gamma_{t}$ the curve $\tau_t (\gamma)$. We are now ready to define a replacement for $P\res \mathcal{N}_i$. We fix a smooth compactly supported function $\psi_i$ in $\mathbb R^{n+1}$ which is positive on $U_i$ and vanishes on $\partial U_i$, a small positive number $\kappa_i$, and we define 
\[
\Sigma^i := \left\{(x,y) : x\in U_i, y \in \gamma_{ \kappa_i \psi_i (x)}\right\} \cap U_i \times B^{n+1}_{\lambda_i}\, .
\]
Choosing $\kappa_i$ sufficiently small we can ensure that the mod 2 current $P^{i} := \Phi_\sharp \Sigma^i$ satisfies $\partial P^i = \partial (P\res \mathcal{N}_i)$. Moreover we can make $\mathbb{F} (P^i - P\res \mathcal{N}_i)$ and $\mathbb M (P^i) - \mathbb M (P\res \mathcal{N}_i)$ smaller than any desired threshold by choosing $\kappa_i$ sufficiently small. Note finally that, clearly, $\Sigma_i$ is smooth in $\mathcal{N}_i$. 

We next enumerate all the $m-1$-dimensional faces $\sigma_i$ of all the $m$-dimensional faces $F^k$ as $\sigma_1, \sigma_2, \ldots , \sigma_N$. We choose our parameters in such a way that the sets $\mathcal{N}_i$ are pairwise disjoint. Our final mod 2 current $P'$ will then be defined to be 
\[
P' := \sum_i P^i + P \res \mathcal{M} \setminus \bigcup_i \mathcal{N}_i\, .
\]
$P'\res \mathcal{M}\setminus B_{\delta_c}(\mathcal{K}^{m-2})$ is then smooth by construction and, choosing the parameters accordingly, $P'$ is homologous to $P$ and we achieve the desired estimates since we can make $\spt (P')$ arbitrarily close to $\spt (P)$, $\mathbb{M} (P')$ arbitrarily close to $\mathbb{M} (P)$, and $\mathbb{F} (P'-P)$ arbitrarily small.

Finally, coming the estimate on $\|P'\| (B_{\delta_c'} (\mathcal{K}^{m-2}))$, the proof follows \emph{verbatim} that of \cite{ABCD}, with $\delta_c'$  chosen small enough just to ensure that $\|P\| (B_{\delta_c'} (\mathcal{K}^{m-2})) \leq \varepsilon_c$.
\end{proof}

\begin{remark}\label{r:m-2corepresentability}
A routine modification of the arguments in the proof of Proposition \ref{p:poly_approx_prescribedsing} implies that the mod 2 cycle $P'$ can be chosen so that its singularities are all {\em contained} in $\mathcal{K}^{m-2}$.
\end{remark}

\subsection{Squeezing lemma}

In the second statement $-$ Proposition \ref{p:squash}, which is the mod 2 version of the \emph{squeezing} lemma in \cite[Section 4]{ABCD} $-$ we are given two $m$-dimensional integral mod 2 cycles $S$ and $R$ which agree outside of a sufficiently small neighborhood of the $m-2$-dimensional skeleton $\mathcal{K}^{m-2}$. We will then show that:

\begin{itemize}
\item $S$ and $R$ represent the same homology class;
\item There is a smooth deformation $R'$ of $R$ which is close, in terms of mass and in flat norm, to $S$;
\item $R'$ coincides with $S$ outside a slightly larger neighborhood of $\mathcal{K}^{m-2}$.
\end{itemize}

\begin{pro}\label{p:squash}
Let $m$ and $\mathcal{M}$ be as in Assumption \ref{a:1} and $\mathcal{K}$ a smooth triangulation of $\mathcal{M}$. Then for every $\varepsilon_d >0$ and every $\eta_d>0$ there exists $\delta_d(\varepsilon_d, \eta_d, \mathcal{K}, \mathcal{M})>0$ with the following property. Suppose $S$ and $R$ are $m$-dimensional integral cycles mod 2 in $\mathcal{M}$ and that 
\[
S\res\mathcal{M}\setminus B_{\delta_d}\left(\mathcal{K}^{m-2}\right)=R\res\mathcal{M}\setminus B_{\delta_d}\left(\mathcal{K}^{m-2}\right)\, .
\]
Then $S$ and $R$ are homologous and moreover there exist a mod 2 integral cycle $R'$ in their homology class and a diffeomorphism $\Phi$ of $\mathcal{M}$ with the following properties:
\begin{enumerate}\itemsep0.2em
\item $\mathbb{M}(R') \leq (1+\varepsilon_d) \mathbb{M}(S)$;
\item $\mathbb{F}(R'-S)\leq C (\varepsilon_d \mathbb{M} (S)+2\|S\| (B_{\eta_d} (\mathcal{K}^{m-2})))^{\frac{m+1}{m}}$, with $C=C (\mathcal{M})$;
\item $R'\res\mathcal{M}\setminus B_{\eta_d}\left(\mathcal{K}^{m-2}\right)=S\res\mathcal{M}\setminus B_{\eta_d}\left(\mathcal{K}^{m-2}\right)$;
\item $\Phi$ is in the isotopy class of the identity and $R'= \Phi_\sharp R$.
\end{enumerate}
\end{pro}

\begin{proof}
The proof follows \emph{verbatim} that of \cite[Proposition 4.3]{ABCD}; for the reader’s convenience, we recall it below.

Without loss of generality we can assume that the $\mathbb{Z}_2$ homology class of $S$ is nontrivial, so that $\mathbb{M} (S)>0$. The conclusion that $R$ and $S$ are homologous follows from the fact that they coincide outside $B_{\delta_d} (\mathcal{K}^{m-2})$. In particular $\spt (S-R)\subset B_{\delta_d} (\mathcal{K}^{m-2})$: since for $\delta_d$ smaller than a constant $c (\mathcal{K})$ the latter has trivial $m$-dimensional homology, it follows that $S-R$ is homologically trivial. 

We now let $\varepsilon_d$ and $\eta_d$ be given as in the statements. We further fix $\bar{\varepsilon}_d$, whose choice will be specified later (but which will depend only on $\varepsilon_d$), and apply Lemma \ref{l:Phi} with $\varepsilon_a=\bar\varepsilon_d$ and $\eta_a=\eta_d$. We therefore get the parameter $\delta_a=:\delta_d$ (which will be the one of the conclusion of the proposition) and, after fixing yet another $\gamma$ (whose choice will now depend on $R$), we get the map $\Phi$ satisfying the requirements of Lemma \ref{l:squash-2}. The requirements (3) and (4) of Proposition \ref{p:squash} are then satisfied by construction.  
Moreover estimate (2) follows from the isoperimetric inequality for mod 2 integral currents and from (1) and (3). Indeed there is a mod 2 integral current $T$ such that $\partial T = S-R'$ and 
\[
\mathbb{M} (T) \leq C \left(\mathbb{M} (S-R')\right)^{\frac{m+1}{m}}
\]
with $C= C (\mathcal{M})$. Using (1) and (3) we then estimate
\begin{align}
\mathbb{M} (S-R') &= \|S-R'\| (B_{\eta_d} (\mathcal{K}^{m-2})) \leq \|S\| (B_{\eta_d} (\mathcal{K}^{m-2}))
+ \|R'\| (B_{\eta_d} (\mathcal{K}^{m-2})) \nonumber\\
&= \|S\| (B_{\eta_d} (\mathcal{K}^{m-2})) + \mathbb{M} (R') - \|S\| (\mathcal{M}\setminus B_{\eta_d} (\mathcal{K}^{m-2}))\nonumber\\
&= 2 \|S\| (B_{\eta_d} (\mathcal{K}^{m-2})) + \mathbb{M} (R') - \mathbb{M} (S)\nonumber\\
&\leq 2 \|S\| (B_{\eta_d} (\mathcal{K}^{m-2})) + \varepsilon_d \mathbb{M} (S)\,.
\end{align}
It remains to prove (1). Note that 
\begin{align}
\mathbb{M} (R') &\leq \mathbb{M} (\Phi_\sharp (R\res B_{\delta_d} (\mathcal{K}^{m-2}))) + \mathbb{M} (\Phi_\sharp (R\res \mathcal{M}\setminus B_{\delta_d} (\mathcal{K}^{m-2})))\nonumber\\
&= \mathbb{M} (\Phi_\sharp (R\res B_{\delta_d} (\mathcal{K}^{m-2}))) + \mathbb{M} (\Phi_\sharp (S\res \mathcal{M}\setminus B_{\delta_d} (\mathcal{K}^{m-2})))\nonumber\\
&\leq \mathbb{M} (\Phi_\sharp (R\res B_{\delta_d} (\mathcal{K}^{m-2}))) + ({\rm Lip}\, \Phi)^m \mathbb{M} (S)\nonumber\\
&\leq \mathbb{M} (\Phi_\sharp (R\res B_{\delta_d} (\mathcal{K}^{m-2}))) + (1+\bar \varepsilon_d)^m \mathbb{M} (S)\, .
\end{align}
Hence we apply Lemma \ref{l:squash-2} and infer 
\[
\mathbb{M} (R') \leq C (1+\bar\varepsilon_d)^{m-2} \gamma^2 \mathbb{M} (R) + (1+\bar \varepsilon_d)^m \mathbb{M} (S)\,.
\]
Next we fix $\bar \varepsilon_d$ so that $(1+\bar\varepsilon_d)^m=1+\frac{\varepsilon_d}{2}$, and then we choose $\gamma$ sufficiently small so that 
\[
C (1+\bar\varepsilon_d)^{m-2} \gamma^2 \mathbb{M} (R) \leq \frac{\varepsilon_d}{2} \mathbb{M} (S)\, .
\]
Note that the choice of $\gamma$, unlike that of $\varepsilon_d$, will indeed depend on $R$ and $S$.
\end{proof}

\section{Proof of the theorem}\label{s:proofofmain}

Before coming to the proof of Theorem \ref{t:1}, we state some preliminary topological results and refer to Appendix \ref{s:Appendix_ThomMO} for an extended discussion. \vspace{0.25cm}

A refined study of the cohomology of the infinite Grassmannian $BO(n)$ allows to prove that the Thom space of the universal $n$-plane bundle $T(\gamma^n)$ is $2n$-equivalent to a product of mod 2 Eilenberg-MacLane spaces; we refer to Appendix \ref{s:Appendix_ThomMO} for the precise definition of the indices and for a proof of Theorem \ref{t:productEM}. In particular, denoting by $Y$ the product of Eilenberg-MacLane spaces given below: $$Y:=K(\mathbb{Z}_2,n) \times K(\mathbb{Z}_2,n+2) \times \dots \times \left(K(\mathbb{Z}_2,n+h)\right)^{d(h)} \times \dots \times \left(K(\mathbb{Z}_2,2n)\right)^{d(n)},$$ we have the following theorem.

\begin{thm}[\protect{\cite[Théorème II.10]{Thom54}}]\label{t:productEM}
There exists a map $H: T(\gamma^n) \rightarrow Y$ which is a $2n$-equivalence, for all positive integers $n$.
\end{thm}

\begin{corollary}\label{c:restriction}
There exists a map $g : K(\mathbb{Z}_2,n)^{2n} \rightarrow T(\gamma^n)$ such that $g^*(u)=\iota$, where $K(\mathbb{Z}_2,n)^{2n}$ denotes the $2n$-skeleton of $K(\mathbb{Z}_2,n)$ and $u, \iota$ the fundamental classes of $T(\gamma^n)$ and $K(\mathbb{Z}_2,n)$ respectively.
\end{corollary}

\begin{remark}\label{r:productEM}\label{r:productEM}
We remark that, in the integral setting, the Thom space of the universal \emph{oriented} $n$-plane bundle is much more complicated, since the equivalent complex is a nontrivial interated fibre bundle with different fibers of type $K(\mathbb{Z}_2,i), K(\mathbb{Z},j)$ or possibly even $K(\mathbb{Z}_p,k)$; this prevents obtaining a map like $g$ in Corollary \ref{c:restriction} simply by restricting the one obtained from the equivalent complex.
\end{remark}

\noindent
In particular, it is possible to prove the following result, which can be seen as a corollary of Theorem \ref{t:productEM} but, since it is fundamental for us in the conclusion of the proof of Theorem \ref{t:1}, we include an independent proof for completeness.

\begin{lem}\label{l:n+2_equivalenza}
Let $h : T(\gamma^n) \rightarrow K(\mathbb{Z}_2,n)$ be a map representing the Thom class $u \in H^n(T(\gamma^n), \mathbb{Z}_2)$. Then $h$ is an $n+2$-equivalence, for all positive integers $n$.
\end{lem}

\begin{proof}
We can assume without loss of generality that $n\geq 2$; indeed for $n=1$ the two spaces are homotopy equivalent since $T(\gamma^1)$ is homotopy equivalent to the infinite real projective space $\mathbb{R}\mathbb{P}^{\infty}$, which is a realization of $K(\mathbb{Z}_2,1)$.

Using the Serre spectral sequences, one can show that the cohomology $H^{n+i}(K(\mathbb{Z}_2,n),\mathbb{Z}_2)$ is generated for $i\leq 2$ by the Steenrod squares $Sq^1$ and $Sq^2$ of the fundamental class, while for any prime $p>2$ the cohomology $H^*(K(\mathbb{Z}_2,n), \mathbb{Z}_p)$ has no generators.

The cohomology ring of $BO(n)$ with coefficients in $\mathbb{Z}_2$ is generated by the Stiefel-Whitney classes $w_1,\dots,w_n$ of $\gamma^n$, \textit{cfr.} \cite[Theorem 7.1]{MilnorStasheff} that is $$H^*(BO(n), \mathbb{Z}_2) = \mathbb{Z}_2\big[w_1,\dots,w_n\big],$$ while, for odd primes $p$, the cohomology $H^*(BO(n), \mathbb{Z}_p)$ is a polynomial ring generated by the (mod $p$) Pontrjagin classes $p_1, \dots, p_{{\lfloor n / 2 \rfloor}}$ of $\gamma^n$, \textit{i.e.} \begin{equation}\label{e:modpcohomology}H^*(BO(n), \mathbb{Z}_p) = \mathbb{Z}_p\big[p_1,\dots,p_{{\lfloor n / 2 \rfloor}}\big]. \end{equation}

For every prime $p$ we denote by $\Phi_p$ the Thom isomorphism between $H^i(BO(n), \mathbb{Z}_p)$ and $H^{n+i}(T(\gamma^n), \mathbb{Z}_p)$ and by $u_p$ the Thom class. Since $Sq^i(u_2)=\Phi_2(w_i)$ and since it is known that from \eqref{e:modpcohomology} one obtains that $\tilde{H}^{n+i}(T(\gamma^n), \mathbb{Z}_p)$ has no generators for $i<n$, it is possible to conclude that for any group coefficient $\mathbb{Z}_p$ the induced map in cohomology $$h^* : H^{n+i}(K(\mathbb{Z}_2,n), \mathbb{Z}_p) \rightarrow H^{n+i}(T(\gamma^n), \mathbb{Z}_p)$$ is an isomorphism for dimensions smaller or equal than $n+1$ and a monomorphism in dimension $n+2$; in fact, for $i=2$ then $H^{n+i}(T(\gamma^n),\mathbb{Z}_2)$ admits in general another generator given by $\Phi_2(w_1^2)$.

Since $K(\mathbb{Z}_2,n)$ and $T(\gamma^n)$ are simply connected, by Theorem \ref{t:II.6} we conclude that $$\pi_k\big(K(\mathbb{Z}_2,n), T(\gamma^n)\big)=0 \quad \text{for }k\leq n+2,$$ and that $h$ is an $n+2$-equivalence for all positive integers $n$, ending the proof.
\end{proof}

\begin{remark}
An obvious consequence of Lemma \ref{l:n+2_equivalenza} is that for dimensions $m\in\{1,2\}$ (and any codimension $n\geq1$) every mod 2 homology class can be represented by a smooth embedded submanifold $\Sigma \subset \mathcal{M}$.
\end{remark}

\noindent
We can now provide a proof of Theorem \ref{t:1}, following the strategy developed in \cite{ABCD}.

\begin{proof}[Proof of Theorem \ref{t:1}] 
Fix $ \varepsilon_c > 0$, whose choice will be specified later, and a mod 2 integral cycle $T$ in a homology class $\tau\in H_m (\mathcal{M}, \mathbb{Z}_2)$. We first apply Proposition \ref{p:poly_approx_prescribedsing} to find a sufficiently small $\delta_c'>0$, a suitable smooth triangulation $\mathcal{K}$ of the manifold and a new mod 2 integral cycle $P'=:S$ with the property that $S$ is in the same homology class of $T$ and the following facts hold:
\begin{itemize}
\item $\mathbb{M} (S) \leq \mathbb{M} (T) + 3\varepsilon_c$ and $\mathbb{F} (S-T) < 3 \varepsilon_c$;
\item $\|S\| (B_{\delta_c'} (\mathcal{K}^{m-2})) \leq 3\varepsilon_c$;
\item $B_{\delta_c'} (\mathcal{K}^{m-2})$ is homotopy equivalent to $\mathcal{K}^{m-2}$;
\item $S\res \mathcal{M} \setminus B_{\delta_c} (\mathcal{K}^{m-2})= \llbracket \Gamma \rrbracket\modtwo$ for a smooth submanifold $\Gamma$.
\end{itemize}
Note that if we first choose $\varepsilon_c$, then $\delta_c$, $\delta_c'$ and $\frac{\delta_c}{\delta_c'}$ can all be made smaller than any desired constant, while the triangulation is instead kept fixed (because it depends only on $\varepsilon_c$). 

We have now fixed a smooth triangulation $\mathcal{K}$ and we can therefore fix constant $C_0$ and $\bar\delta$ so that Lemmas \ref{l:spigoli} and \ref{l:spigoli-allisciati} apply. We now require that $V_{{\delta}_c/2}(\mathcal{K}^{m-2})  \subset \subset  B_{\delta_c'} (\mathcal{K}^{m-2})$ for some ${\delta}_c/2 << \delta_c'$.  Hence we apply Lemma \ref{l:spigoli-allisciati} (where $\delta'<\delta$ corresponds here to $\delta_c/2<\tilde{\delta}/2)$ to find a $U_{\tilde{\delta}/2} (\mathcal{K}^{m-2})$ suitably close to $V_{\tilde{\delta}/2} (\mathcal{K}^{m-2})$. We will want that $B_{\delta_c}(\mathcal{K}^{m-2}) \subset U_{\tilde{\delta}/2} (\mathcal{K}^{m-2}) \subset V_{\tilde{\delta}/2} (\mathcal{K}^{m-2}) \subset V_{\tilde{\delta}} (\mathcal{K}^{m-2}) \subset B_{\delta_c'} (\mathcal{K}^{m-2})$. This step requires to take $\frac{\delta_c}{\delta_c'}$ sufficiently small and $\tilde \delta < \delta'_c$. Define now $\Omega := \mathcal{M} \setminus U_{\tilde{\delta}/2} (\mathcal{K}^{m-2})$.

The current $\llbracket \Gamma \rrbracket\modtwo$ obtained from Proposition \ref{p:poly_approx_prescribedsing} is (when restricted to $\Omega$ and not relabelled) a smooth compact submanifold of $\Omega$ with $\partial \Gamma \subset \partial \Omega$, provided $\partial \Omega$ is transversal to $\Gamma$, which can be ensured via a small smooth perturbation. Denoting by $x \in H^{n}(\mathcal{M})$ the Poincaré dual of $\tau$, note that its restriction $x_{|\Omega} \in H^n(\Omega)$ to $\Omega$ is the relative Poincaré dual of a relative homology class which is represented by the smooth compact embedded submanifold $\Gamma \subset \Omega$ with boundary $\partial \Gamma = \Gamma \cap \partial \Omega$. Hence, by Theorem \ref{t:Thomboundary}, there exists a map $$F: \Omega \rightarrow T(\gamma^n)$$ such that $F^*(u) = x_{|\Omega}$; in addition, $F$ is smooth and transverse on $BO(n)$ (and such that $F_{|\partial \Omega}$ is also transverse), so that $F^{-1}(BO(n))=\Gamma$, which is the smooth part of $S$. 

We then take $\delta$ sufficiently small so that $\Omega \subset \mathcal{M}\setminus U_\delta (\mathcal{K}^{m-n-1})$ for the $U_\delta$ given in Lemma \ref{l:spigoli-allisciati}. Then, by Lemma \ref{l:co-skeleton} we have that $\mathcal{M} \setminus U_\delta(\mathcal{K}^{m-n-1})$ is homotopy equivalent to a complex of dimension $2n$. Denote $$Q:=\mathcal{M} \setminus U_\delta(\mathcal{K}^{m-n-1}).$$

\noindent
Given the $n$-dimensional cohomology class $x \in H^{n}(\mathcal{M})$ which is the Poincaré dual of $\tau$, we consider its restriction to $Q$, that is $x_{|Q} \in H^n(Q)$; note that $x_{|Q}$ can be represented by a continuous map $$c:Q \rightarrow K(\mathbb{Z},n)$$ (in a suitable homotopy class of continuous maps) pulling-back the fundamental class of $K(\mathbb{Z},n)$ to itself, \textit{i.e.} $c^*(\iota)=x_{|Q}.$ By Corollary \ref{c:restriction} and cellular approximation, there exists a map $f: Q \rightarrow T(\gamma^n)$ such that the diagram commutes, \textit{i.e.} $f$ pulls-back the universal Thom class to $x_{|Q}$.

\[
\xymatrix{
 & T(\gamma^n) \ar[d]^{h} \\
Q \ar[r]_{c \quad } \ar[ur]^{f} & K(\mathbb{Z}_2,n) \ar@/_2pc/@{.>}[u]_{g}
}
\]

By the same construction of the second part of the proof of Theorem \ref{t:Thomboundary}, we can assume without loss of generality that $f$ is smooth throughout $Q \setminus f^{-1}(U(\infty))$ and transversal to (a sufficiently high dimensional approximation of) the zero cross-section $BO(n) \subset T(\gamma^n)$, with $\partial f$ also transversal to it. Hence, $f^{-1}(BO(n))$ is a compact smooth $m$-dimensional embedded submanifold, with boundary contained in $\partial Q$; denote it as $$\mathcal{N}:=f^{-1}(BO(n)).$$ Moreover, $\mathcal{N}$ represents the relative Poincaré dual of $x_{|Q}$, which equals $j_*(\tau) \in H_m(Q,\partial Q),$ where $j_*:H_m(\mathcal{M}) \rightarrow H_m(Q,\partial Q).$ 

We next extend $\llbracket\mathcal{N}\rrbracket\modtwo$ (which is an integral current mod 2 in $\mathcal{M}$) to a mod 2 integral current $N$ with the property that $N\res \mathcal{M}\setminus \mathcal{K}^{m-n-1}$ is induced by a smooth submanifold and $N\res Q = \llbracket \mathcal{N}\rrbracket\modtwo$. First of all, because $\mathcal{N}$ is transversal to $\partial Q$, we can extend it to a smooth submanifold over the union $Q'$ of $Q$ with any smooth collaring extension of $\partial Q$. We can then use Lemma \ref{l:spigoli-allisciati} to find such an extension $Q'$ (which consists of $Q\cup \mathcal{C}$, where $\mathcal{C}$ is the smooth tubular neighborhood in Lemma \ref{l:spigoli-allisciati}) containing $\mathcal{M}\setminus V_{\delta'} (\mathcal{K}^{m-n-1})$ for some $\delta'<\delta$ positive. Since $\mathcal{N}$ intersects $\partial Q$ transversally, we can extend it to a smooth submanifold of $Q'$ with boundary in $\partial Q'$, meeting $\partial Q'$ transversally. With abuse of notation this extension is still denoted by $\mathcal{N}$. We can now use the map $\Phi$ of Lemma \ref{l:second-Phi} and set $$N:=\Phi_\sharp \llbracket \mathcal{N}\rrbracket\modtwo.$$ The latter current is mod 2 integer rectifiable because $\Phi$ is Lipschitz (and, in particular, $N$ has finite mass). Given that $\Phi$ is a diffeomorphism over $\mathcal{M}\setminus \Phi^{-1} (\mathcal{K}^{m-n-1}) \subset \mathcal{M}\setminus V_{\delta'} (\mathcal{K}^{m-n-1})$, then $N\res \mathcal{M}\setminus \mathcal{K}^{m-n-1} = \llbracket \Sigma \rrbracket\modtwo$ for some smooth submanifold $\Sigma$. Moreover $\spt (\partial N)\subset \mathcal{K}^{m-n-1}$ and in particular, by Federer flatness theorem, $\partial N = 0$, namely $N$ is a cycle. 

Consider now the two maps $F: \Omega \rightarrow T(\gamma^n)$ and $f: Q \rightarrow T(\gamma^n)$ such that $F^{-1}(BO(n))= \Gamma$ and $f^{-1}(BO(n))= \mathcal{N}\cap Q$. If we consider the restriction of $f$ to $\Omega \subset Q$, we obtain a new map $f|_{\Omega}: \Omega \rightarrow T(\gamma^n)$ that pulls-back the universal Thom class to $x_{|\Omega}$. By Lemma \ref{l:co-skeleton} we observe that $\Omega$ has the homotopy type of an $n+1$-complex, so that by Lemma \ref{l:n+2_equivalenza} and Corollary \ref{c:Whitehead_Homotopyclasses} we can conclude that $F$ and $f_{|\Omega}$ are homotopic: the homotopy can be taken smooth by \cite[Lemme IV.5]{Thom54}. In particular, we define the smooth homotopy $H:[0,1] \times \Omega$ such that $H(0,x)= f_{|\Omega}(x)$ and $H(1,x)= F(x)$. In a small collar neighborhood $\mathcal{C}$ of $\partial \Omega$ inside $\Omega$, which we identify with $\partial \Omega\times (0,1]$, we then glue the maps $f$ and $F$ together. Using the notation $x=(y,s)\in \mathcal{C}$ and after defining a smooth function $\varphi$ on $[0,1]$ which is identically equal to $0$ in a neighborhood of $0$ and identically equal to $1$ in a neighborhood of $1$, we set  
\begin{equation}\label{e:homotopy}
\widehat{f}(x):=\begin{cases}
       F(x)  & \quad if  \, \, x\in \Omega\setminus \mathcal{C},\\    
       H\left(x,\varphi (s)\right)  & \quad if\,\,  x\in \mathcal{C},   \\
        f(x)  & \quad if\,\, x\in Q\setminus \Omega  
\end{cases}\end{equation}

Since $T(\gamma^n) \setminus \{\infty\}$ is a smooth manifold, it follows from \cite[Proposition 2.3.4 (ii)]{Wall} that we can find $\widehat{f}:Q\rightarrow T(\gamma^n)$, not relabelled, which is smooth throughout $Q \setminus f^{-1}(U(\infty))$, coincides with $f(x)$ in a neighborhood of $\partial Q$ and with $F$ on $\mathcal{M}\setminus V_{\tilde{\delta}} (\mathcal{K}^{m-2})$; the approximation can be taken close enough so that $\widehat{f}$ is in the same homotopy class.
Analogously, by \cite[Proposition 4.5.10]{Wall}, we can perturb $\widehat{f}$ so that it is transverse to $BO(n)$ and coinciding with $f(x)$ in a neighborhood of $\partial Q$ and with $F$ on $\mathcal{M}\setminus V_{\tilde{\delta}} (\mathcal{K}^{m-2})$.

Consider now the submanifold $\Sigma'$ of $\mathcal{M}\setminus \mathcal{K}^{m-n-1}$ which consists of:
\begin{itemize}
    \item $\Sigma$ in $V_{\delta'} (\mathcal{K}^{m-n-1})\setminus \mathcal{K}^{m-n-1}$;
    \item $\mathcal{N}$ on $U_\delta (\mathcal{K}^{m-n-1})\setminus V_{\delta'} (\mathcal{K}^{m-n-1})$;
    \item $\widehat{f}^{-1} (BO(n))$ on $Q$.
\end{itemize}
This is a smooth submanifold because:
\begin{itemize}
    \item $f$ and $\widehat{f}$ coincide in a neighborhood of $\partial Q$ and hence $\widehat{f}^{-1} (BO(n))$ coincides with $\mathcal{N}$ in a neighborhood of $\partial Q$;
    \item $\Sigma = \Phi (\mathcal{N}) = \mathcal{N}$ in a neighborhood of $\partial V_{\delta'} (\mathcal{K}^{m-n-1})$.
\end{itemize}
    Moreover, $R= \llbracket \Sigma'\rrbracket\modtwo \in \mathbf{R}_m(\mathcal{M},\mathbb{Z}_2)$ and $\spt (\partial R)\subset \mathcal{K}^{m-n-1}$; in particular it is a cycle by Federer's flatness theorem. Observe also that $R-S$ is supported, by construction, in $V_{\tilde{\delta}} (\mathcal{K}^{m-2})$, which is homotopy equivalent to $\mathcal{K}^{m-2}$, and thus has trivial $m$-homology. In particular, $R$ and $S$ belong to the same homology class. 

The conclusion now follows \emph{verbatim} as in \cite{ABCD}, but we recall the last passages here for completeness: we apply Proposition \ref{p:squash} to $S$ and $R$, noticing that the $\varepsilon_d$ in Proposition \ref{p:squash} is a parameter to be chosen in terms of the $\varepsilon$ of the statement of Theorem \ref{t:1}, and the $\eta_d$ in Proposition \ref{p:squash} is $\delta'_c$ here. This gives us a parameter $\delta_d$, which depends on $\varepsilon_d$ and $\delta'_c$. In turn we impose that $\tilde{\delta} \leq \delta_d$ so that we can apply Proposition \ref{p:squash}. Since $\varepsilon_d$ will be specified only in terms of $\mathbb{M}(T)$ and of $\varepsilon$ in the statement of Theorem \ref{t:1}, while $\delta'_c$ depends on $\varepsilon_c$, which will also be specified only in terms of $\mathbb{M}(T)$ and $\varepsilon$ in the statement of Theorem \ref{t:1}, the parameter $\tilde{\delta}$ can be taken smaller than $\delta_d$. We can then find a current $R':=\Phi_\sharp R$ for a smooth diffeomorphism $\Phi$ isotopic to the identity such that  
\begin{align}
\mathbb{M} (R') \leq (1+\varepsilon_d)\, \mathbb M (S) \leq (1+\varepsilon_d) (\mathbb{M} (T) + 3\varepsilon_c)\,\nonumber .
\end{align}
We therefore conclude that $R'$ is homologous to $R$, hence to $S$, and therefore to $T$. Moreover, if we choose
\begin{align}
&\varepsilon_d \,(3+\mathbb{M} (T)) < \frac{\varepsilon}{2} \qquad \mbox{and} \qquad
\,3\varepsilon_c < \frac{\varepsilon}{2}\,\nonumber ,
\end{align}
then $\mathbb{M} (R') \leq \mathbb{M} (T) + \varepsilon$. Finally
\begin{align}
\mathbb{F} (T-R') &\leq 3\varepsilon_c + \mathbb{F} (S-R') 
\leq 3\varepsilon_c + C (\varepsilon_d
\,\mathbb{M} (S) + 2\|S\| (B_{\delta'_c}(\mathcal{K}^{m-2})))^{\frac{m+1}{m}}\nonumber\\
&\leq 3\varepsilon_c + C (\varepsilon_d (\mathbb{M} (T) +\varepsilon_c) + 6\varepsilon_c)^{\frac{m+1}{m}}\,\nonumber.
\end{align}
Therefore it is clear that a suitable choice of $\varepsilon_d$ and $\varepsilon_c$ depending only on $\mathbb{M} (T)$ and $\varepsilon$ suffices to show $\mathbb{F} (T-R') \leq \varepsilon$. 

The proof of part $(3)$ of Theorem \ref{t:1} is analogous; by assumption we know that $\tau$ is represented by a smooth submanifold $\Sigma$ and hence, by Theorem \ref{t:Thomclosed} there exists a map $\ell:\mathcal{M} \rightarrow T(\gamma^n)$ which pulls-back the universal Thom class $u \in H^n(T(\gamma^n), \mathbb{Z}_2)$ to the Poincaré dual of $\tau$. Substituting in the previous steps the map $f$ with this new map $\ell$, defined over the whole ambient space $\mathcal{M}$, and defining a similar homotopy as that one in \eqref{e:homotopy}, the result follows by applying Proposition \ref{p:squash} to $S$ and $\llbracket \Sigma \rrbracket\modtwo$, where $S$ is the integral cycle mod 2 denoted $P'$ in Proposition \ref{p:poly_approx_prescribedsing}.
\end{proof}

\section{Optimality of the construction}\label{s:optimality}

In this section we show that the construction of Theorem \ref{t:1} is optimal, in the sense that there exists $m$-dimensional $\mathbb{Z}_2$ homology classes that cannot be represented by mod 2 cycles which are smooth embedded manifolds in the complement of an $m-n-2$-dimensional skeleton of any smooth triangulation of $\mathcal{M}$. In fact, proving sharpness of the construction becomes even subtler in mod 2 homology than in the integral setting since singularities that appear in this context all arise from the impossibility of finding embeddings in low codimensions, and not $-$ as in integral classes $-$ by innate singularities obstructing also Steenrod representability. For this reason, we cannot exploit Sullivan's geometric theory of resolution of singularities by blow-up in \cite{SullivanLiverpool}, as done in \cite[Theorem 6.3]{ABCD} to derive a contradiction. Instead, we need to rely on singularity theory of stable mappings, coupled with an elegant result due to Grant and Sz\H{u}cs \cite{Grantszucs} that characterizes their singular set; we refer to \cite{GGbook} for an accessible overview on the theory of stable mappings. \vspace{0.25cm}

The mod 2 homology class of least dimension that cannot be represented by a smooth embedded submanifold is a 4-dimensional homology class in a 6-dimensional closed smooth manifold $\mathcal{M}$, and it comes as a corollary of Teichner's construction of $S^2$-bundles over $m$-dimensional base manifolds (for every $m\geq4$) such that every 2-dimensional $\mathbb{Z}_2$ cohomology class restricting to the generator in the fibre cannot be written as the second Stiefel-Whitney class $w_2(E)$ of some real vector bundle $E$ over the base manifold (recall that $w_2(E)^2=p_1(E)$ (mod 2), where $p_1$ is the first Pontrjagin class), see \cite[Section 3]{Teichner95}.

In fact, let $h$ be the canonical map from $T(\gamma^2)$ to $K(\mathbb{Z}_2,2)$; by the theory of Thom \cite{Thom54}, there is an obstruction in extending the homotopy inverse $g$ of $h$ from the $4$-skeleton to the 5-skeleton given by the Eilenberg-MacLane invariant associated to the second non-null homotopy group $\pi_4(T(\gamma^2))$, which is an element in $H^5(K(\mathbb{Z}_2,2))$ with coefficients in $\pi_4(T(\gamma^2))$ generating the kernel of the homomorphism $$h^*:H^5(K(\mathbb{Z}_2,2), \mathbb{Z}) \rightarrow H^5(T(\gamma^2)), \mathbb{Z}).$$ Precisely, the kernel of $h^*$ is generated by $(1/2)$$\beta$$\textbf{p}(\iota)$, where $\textbf{p}$ is the Pontrjagin square $\textbf{p}: H^*(\mathcal{M}, \mathbb{Z}_2) \rightarrow H^{2*}(\mathcal{M}, \mathbb{Z}_4)$, \emph{cfr.} \cite{Pontrjaginsq}, and $\beta$ is the mod 2 Bockstein $\beta : H^*(\mathcal{M}, \mathbb{Z}_2) \rightarrow H^{*+1}(\mathcal{M}, \mathbb{Z})$ associated to the exact sequence $$0 \rightarrow \mathbb{Z}  \rightarrow \mathbb{Z}  \rightarrow \mathbb{Z}_2  \rightarrow 0.$$

Hence, if a mod 2 homology class $\tau \in H_m(\mathcal{M}, \mathbb{Z}_2)$ admits a smooth embedded representative, then $\beta (x^2)=0$, where $x \in H^n(\mathcal{M}, \mathbb{Z}_2)$ is the Poincaré dual of $\tau$. In particular, Teichner's construction provides an orientable 6-dimensional closed manifold $\mathcal{M}$ with $\bar{x} \in H^2(\mathcal{M},\mathbb{Z}_2)$ such that $\beta (\bar{x}^2)\neq0,$ see \cite[Lemma 2]{Teichner95}. It is a finer and very elegant result due to Grant and Sz\H{u}cs in \cite{Grantszucs} that the cohomology operation $\beta (\bar{x}^2)\neq0$ also obstructs realizability by immersions and, more precisely, they characterize the obstruction $\beta (\bar{x}^2) \in H^{5}(\mathcal{M}, \mathbb{Z})$ as the integral class whose Poincaré dual is realized by the singular set of any \emph{stable} map realizing the Poincaré dual of $\bar{x}$.

In particular, recall that a smooth map $f: \Sigma \rightarrow \mathcal{M}$ of a smooth closed $m$-dimensional manifold $\Sigma$ into a smooth $m+n$-dimensional manifold $\mathcal{M}$ is called \emph{stable} if for any sufficiently close smooth map $f^{\prime}: \Sigma \rightarrow \mathcal{M}$ there exist diffeomorphisms $g: \Sigma \rightarrow \Sigma$ and $h: \mathcal{M} \rightarrow \mathcal{M}$ such that $h^{-1} \circ f^{\prime} \circ g=f$. In general, it is a remarkable result due to Thom and Levine \cite{Thomlevine} that stable maps are not always dense\footnote{In the usual $C^{\infty}$ Whitney topology on $C^{\infty}(\Sigma, \mathcal{M})$.} in the space of smooth mappings $C^{\infty}(\Sigma, \mathcal{M})$ but, in codimension $n=2$, stable maps are dense in $C^{\infty}(\Sigma, \mathcal{M})$ whenever $m<21$. Hence, (the Poincaré dual of) Teichner's cohomology class $x \in H^2(\mathcal{M}, \mathbb{Z}_2)$ can be represented by a stable map $\bar{f}: \Sigma \rightarrow \mathcal{M}$, simply by approximating the continuous map given by Steenrod representability by a smooth map, and then by a stable map. The \emph{singular set} of $\bar{f}:\Sigma \rightarrow \mathcal{M}$ is defined as $$\mathcal{S}(\,f):=\{x \in M \mid \operatorname{rank}\, d f(x) <m\} \subseteq M,$$ where $d \bar{f}(x)$ is the differential of $\bar{f}$. The singular set $\mathcal{S}(\,\bar{f})$ is the closure in $M$ of the top singularity stratum $\mathcal{S}^{1,0} \subseteq M$ of \emph{cross-cap} points, singularities that persist under small perturbations, and it is an open dense subset of $\mathcal{S}(\,\bar{f})$. In addition, since $x \in H^2(\mathcal{M},\mathbb{Z}_2)$, then $\mathcal{S}(\,f)$ has codimension $3$ in $\Sigma$, and hence it is of dimension 1; in particular, there is no other stratum of singularities, since all other strata have codimension at least $6$ in $\Sigma$, \emph{cfr.} \cite{Boardman} or \cite[Chapter VI]{GGbook}; moreover, $\mathcal{S}:=\mathcal{S}(\,\bar{f})$ carries a fundamental class $[\mathcal{S}] \in H_{1}(\mathcal{S}, \mathbb{Z}_{\mathcal{S}^{1,0}})$ with twisted coefficients $\mathbb{Z}_{\mathcal{S}^{1,0}}$ according to the tangent bundle of $\mathcal{S}^{1,0}$.

In \cite{Grantszucs}, Grant and Sz\H{u}cs showed that $\beta (\bar{x}^2) \in H^{5}(\mathcal{M}, \mathbb{Z})$ can be interpreted in terms of the singular set $\mathcal{S}$ of any stable map realizing $\bar{x}$; more formally, denoting by $i: \mathcal{S} \rightarrow M$ the inclusion, and by $\tilde{f}=\bar{f} \circ i: \mathcal{S} \rightarrow \mathcal{M}$ the restriction of $\bar{f}$ to its singular set, they proved the following.

\begin{pro}[\protect{\cite[Proposition 4.2]{Grantszucs}}]\label{p:Grantszucs}
Let $x \in H^n(\mathcal{M},\mathbb{Z}_2)$ be realized by a stable map $\bar{f}: \Sigma^{m} \rightarrow \mathcal{M}^{m+n}$ of closed smooth manifolds, where $n$ is even. Then $\beta(x^2) \in H^{2n+1}(\mathcal{M}, \mathbb{Z})$ is the cohomology class in $\mathcal{M}$ realized by the singular set $\mathcal{S}$ of $\bar{f}$. In other words, $\beta(x^2)$ is the Poincaré dual of $$\tilde{f}_*[\mathcal{S}] \in H_{m-n-1}(\mathcal{M}, \mathbb{Z}_{\mathcal{M}}).$$
\end{pro}

\noindent
Hence, we obtain the following.

\begin{thm}\label{t:sharp}
Let $z \in H_4(\mathcal{M},\mathbb{Z}_2)$ be the codimension 2 homology class of Teichner's example in \cite[Lemma 2]{Teichner95} and fix a smooth triangulation $\mathcal{K}$ of $\mathcal{M}$. Then it is impossible to find a representative $\Sigma$ for $z$ which is a smooth embedded submanifold in the complement of the $0$-dimensional skeleton of $\mathcal{K}$.
\end{thm}

\begin{proof}
Recall that, by Steenrod representability, $z$ can be represented by a continuous map $f:\Sigma \rightarrow \mathcal{M}$. We can now approximate $f$ with a homotopic smooth map and, again, by a homotopic stable map, given that for $n=2$ and $m=4$ stable maps are dense (in the $C^{\infty}$ Whitney topology) in the space of smooth maps $C^{\infty}(\Sigma, \mathcal{M})$. Hence, $\bar{x}$ is realized by a stable map $\bar{f}: \Sigma^{m} \rightarrow \mathcal{M}^{m+n}$ of closed smooth manifolds. In addition, by Teichner's construction in \cite[Lemma 1]{Teichner95} we know that $\beta (\bar{x}^2) \neq 0$, where $\bar{x}$ denotes the Poincaré dual of $z$.

Therefore, by \cite[Proposition 4.2]{Grantszucs} we have that $\beta (\bar{x}^2)$ is realized by the singular set $\mathcal{S}$ of $\bar{f}$, which has dimension 1; it is therefore impossible to remove $\tilde{f}(\mathcal{S})$ by removing the $0$-dimensional skeleton of $\mathcal{K}$, which is a finite collection of points.
\end{proof}

\begin{remark}
Note that the proof of Theorem \ref{t:sharp} also shows that it is impossible to find a representative $\Sigma$ for $z$ which is a smooth \emph{immersed} submanifold in the complement of the $0$-dimensional skeleton $\mathcal{K}^0$ of $\mathcal{K}$.
\end{remark}

\newpage
\addtocontents{toc}{\protect\setcounter{tocdepth}{1}}

\appendix

\section{Homotopy type of $T(\gamma^n)$}\label{s:Appendix_ThomMO}

In this appendix we recall Thom's algebraic computations in the study of the homotopy type of $T(\gamma^n)$, which in turn are based on fundamental contributions of Serre, \emph{cfr.} \cite{Serre53}. In particular we highlight the main steps in the proof of Theorem \ref{t:productEM}, \emph{cfr.} \cite[Chapitre II, Section 6]{Thom54}.\vspace{0.25cm}

Recall that the cohomology $H^*(T(\gamma^n),\mathbb{Z}_2)$ is isomorphic to the ideal $J$ of the algebra $H^*(BO(n), \mathbb{Z}_2)$ generated by the Stiefel-Whiteny classes $w_n$. On the other hand, introducing $n$ variables $t_i$ we obtain a basis of $H^h(BO(n), \mathbb{Z}_2)$ generated by symmetrised monomials \begin{equation}\label{e:monomials} \sum(t_1)^{a_1}(t_2)^{a_2} \ldots (t_r)^{a_r},\end{equation} where the sum of the integers $a_i$ is equal to $h$, and the symmetrisation sign $\sum$ means the summation over all \emph{essential permuations} for \eqref{e:monomials}\footnote{\emph{i.e.} over representatives of classes of the symmetric group of $n$ variables modulo the subgroup of permutations fixing the monomial \eqref{e:monomials}.}. Hence, one obtains a basis for the dimension $n+h$ of the ideal $J$ by multiplying the elements of the basis in \eqref{e:monomials} by $w_n=t_1 \cdot t_2 \cdot \ldots t_n$, obtaining all symmetrised monomials \begin{equation}\label{e:monomialsdue}\sum(t_1)^{a_1+1}(t_2)^{a_2+1} \ldots(t_r)^{a_r+1} t_{r+1} \ldots t_n.\end{equation} Indeed, any essential permutation for the monomial \eqref{e:monomials} is essential for \eqref{e:monomialsdue}, and vice versa.

Let $P$ be an arbitrary polynomial in the variables $t_j$. A variable $t_j$ is a called \emph{dyadic} for the polynomial $P$ if the exponent of this variable in terms of the polynomial $P$ is either zero or $2^i$ (with $i=0$ included). Any dyadic variable $t_j$ for the polynomial $P$ is dyadic for $Sq^i P$ as well, where $Sq^i$ denotes the Steenrod squares. By a \emph{non-dyadic factor} of the monomial $(t_1)^{a_1}(t_2)^{a_2} \ldots(t_r)^{a_r}$ we mean the submonomial consisting of all non-dyadic variables; denote the number of these variables by $u$, and by $v$ their total degree. We define a preorder\footnote{A binary relation satisfying reflexivity and transitivity only.} $\succsim$ on the set of monomials in $(t_j)$ as follows: given two monomials $x,y$ we say that $x\succsim y$ if $u(x)>u(y)$ or if $u(x)=u(y)$ and $v(x)<v(y)$. For any number $h \leq n$ consider the classes \begin{equation}\label{e:partizionidiadiche}x_\omega^h=\sum(t_1)^{a_1+1}(t_2)^{a_2+1} \ldots(t_r)^{a_r+1}\cdot t_{r+1} \ldots t_n,\end{equation} where $\omega=\{a_1, a_2, \ldots, a_r\}$ is an arbitrary decomposition of $h$ into summands, with no summand of type $2^i-1$ (non-dyadic decomposition of $h$); we denote the number of such decompositions by $d(h)$.

For any dimension $i \leq h$, consider the classes \begin{equation}\label{e:linindep}x_{\omega_i}^i,\,\, Sq^1 x_{\omega_{i-1}}^{i-1},\,\, S q^2 x_{\omega_{i-2}}^{i-2},\,\, \ldots,\,\, S q^{I_h} x_{\omega_h}^h,\,\, \ldots,\,\, S q^I w_k,\end{equation} where $S q^{I_h}$ is an admissible sequence of total degree $i-h$, and $\omega_h$ is a non-dyadic decomposition of $h$ as in \eqref{e:partizionidiadiche}; all classes in \eqref{e:linindep} are in fact linearly independent. In particular, it is possible to prove that the cohomology classes in \eqref{e:linindep} form a basis of $H^{n+i}(T(\gamma^n))$. We associate to each class $x^h_\omega$ a map $$H_\omega: T(\gamma^n) \rightarrow K(\mathbb{Z}_2,n+h)$$ such that $H^*_{\omega}(\iota)=x^h_\omega$, where $\iota$ is the fundamental class of $K(\mathbb{Z}_2,n+h)$. The maps $H_\omega$ define a map $H$ from $T(\gamma^n)$ into the following product of Eilenberg-MacLane spaces $$Y:=K(\mathbb{Z}_2,n) \times K(\mathbb{Z}_2,n+2) \times \dots \times \left(K(\mathbb{Z}_2,n+h)\right)^{d(h)} \times \dots \times \left(K(\mathbb{Z}_2,2n)\right)^{d(n)}.$$

Since the classes \eqref{e:linindep} form a basis of $H^{n+i}(T(\gamma^n))$, then the homomorphism $H^*$ induced by $H$ is an isomorphism from $H^{n+i}(Y , \mathbb{Z}_2)$ to $H^{n+i}(T(\gamma^n))$ for all $i \leq n$. In mod $p$ coefficients, with $p>2$, the cohomology of $Y$ is trivial and the cohomology of $T(\gamma^n)$ is trivial in dimensions strictly less than $2n$. Hence, $H^*$ is again an isomorphism in dimensions strictly less than $2 n$ and a monomorphism in dimension $2 n$. Thus, for $T(\gamma^n)$ and $Y$ one can apply Theorem \ref{t:II.6}, obtaining that there exists an inverse map $g$ from the $2n$-skeleton of $Y$ to $T(\gamma^n)$ such that $g \circ H$ is homotopic to the identity on the $2n-1$-skeleton of $T(\gamma^n)$. This proves Theorem \ref{t:productEM}.

\end{document}